\documentclass[11pt]{amsart}
\usepackage{amssymb,latexsym,amsmath,amsfonts,amsthm}
\usepackage{color}
\usepackage{soul}
\numberwithin{equation}{section}

\theoremstyle{plain}
\newtheorem{theorem}{Theorem}[section]

\newtheorem{proposition}{Proposition}[section]
\newtheorem{corollary}{Corollary}[section]
\newtheorem{lemma}{Lemma}[section]

\theoremstyle{remark}
\newtheorem*{theorem*}{Theorem}
\newtheorem*{question*}{Question}
\newtheorem{remark}{Remark}[section]
\newtheorem{definition}{Definition}[section]

\usepackage{fullpage}

\newenvironment{psmallmatrix}
 {\left(\begin{smallmatrix}}
 {\end{smallmatrix}\right)}

\newcommand{\overbar}[1]{\mkern 1.5mu\overline{\mkern-1.5mu#1\mkern-1.5mu}\mkern 1.5mu}

\usepackage{mathrsfs} 
\usepackage{fourier}
\usepackage{times}

\title{O\lowercase{perators} \lowercase{in the} C\lowercase{owen}-D\lowercase{ouglas} \lowercase{class and related topics}}
\author{G\lowercase{adadhar} M\lowercase{isra}}
\email{gm@iisc.ac.in}
\address{Department of Mathematics, Indian Institute of Science, Bangalore 560 012}
\thanks{The author gratefully acknowledges the financial support from the Department 
of Science and Technology in the form of the~J~C~Bose National Fellowship and  from the University Grants Commission, Centre for Advanced Study.}
\begin{document}
\maketitle

\baselineskip 14pt


Linear spaces with an Euclidean metric are ubiquitous in mathematics, 
arising both from quadratic forms and inner products.  Operators on 
such spaces also occur naturally. 
In recent years, the study of multivariate operator theory has made substantial progress.  Although, the study of self adjoint operators  goes back a few decades, the non-self adjoint theory has developed at a slower pace. While several approaches to this topic has been developed, the one that has been most fruitful is clearly the study of Hilbert spaces that are modules over natural function algebras like $\mathcal A({\Omega})$, where $\Omega \subseteq \mathbb C^m$ is a bounded domain,  consisting of complex valued functions which are holomorphic on some open set $U$ containing $\overbar{\Omega}$, the closure of $\Omega$. \index{Hilbert module}
The book \cite{misra-DP} showed how to recast many of the familiar theorems of operator theory  in the language of Hilbert modules. The books \cite{misra-EP}  and  \cite{misra-CG} provide an account   of the achievements from the recent past.
The impetus for  much of what is described below comes from the interplay
of operator theory with other areas of mathematics like 
complex geometry and representation theory of locally compact 
groups.  \index{locally compact groups}

\section{Introduction} 

The first half of this expository article describe several  elementary properties of the operators in the Cowen-Douglas class. This is divided into five separate themes. In the second half of the article, we elaborate a little more on each of these themes. 
\index{Cowen-Douglas class}

\subsection{Operators in the Cowen-Douglas class} 
In the paper \cite{misra-CD}, Cowen and Douglas initiated a systematic study of a class of  bounded linear operators  on a complex separable Hilbert space possessing an open set  of eigenvalues of constant (and finite) multiplicity. Let $\Omega$ be the set of eigenvalues of the operator $T:\mathcal H \to \mathcal H$ in this class. Assuming that  $\mbox{\rm ran}(T-wI) =\mathcal H$,  using elementary Fredholm theory, they prove:
For a fixed but arbitrary $w_0\in \Omega$, there is an open neighbourhood $U$ of $w_0$ and  holomorphic functions 
$$\gamma_i:U \to \mathcal H,\,\,  (T-w)\gamma_i(w)=0,\,\, 1\leq i \leq n,$$ such that the  vectors $\{\gamma_1(w), \ldots , \gamma_n(w)\}$ are linearly independent, $w\in U$. They also show that such an operator $T$ defines a holomorphic Hermitian vector bundle $E_T$:\index{holomophic Hermitian vector bundle}
$$t: \Omega \to \mbox{\rm Gr}(n, \mathcal H),\,\, t(w) = \ker (T-w) \subseteq \mathcal H.$$  \index{holomorphic frame}
\!\! This means, for any fixed but arbitrary point $w_0\in \Omega$,  there exists a holomorphic map $\boldsymbol \gamma_T$ of the form $\boldsymbol \gamma_T(w)  = (\gamma_1(w), \ldots , \gamma_n(w))$, $(T-w)\gamma_i(w)=0$ in some open neighbourhood  $U$ of $w_0$. It is called a holomorphic frame for the operator $T$.
Finally, Cowen and Douglas also assume that the linear span of $\{\gamma_1(w), \ldots, \gamma_n(w):w\in \Omega\}$ is dense in $\mathcal H$. 
Let $B_n(\Omega)$ denote this class of operators.

\index{curvature}\index{second fundamental form}\index{line bundle}
One of the striking results of Cowen and Douglas says that 
there is a one to one correspondence between the unitary equivalence class of the operators $T$ and the (local) equivalence classes of the holomorphic Hermitian vector bundles $E_T$ determined by them. 
As a result of this correspondence set up by the Cowen-Douglas theorem, the invariants the vector bundle $E_T$  like the curvature, the second fundamental form, etc. now serve as unitary invariants for the operator $T$. Although finding a complete set of tractable invariants, not surprisingly, is much more challenging. Examples were given in \cite[Example 2.1]{misra-MSR} to show that the class of the curvature alone does not determine the class of the vector bundle except in the case of a line bundle. Before we consider this case in some detail, let us recall the interesting notion of a spanning section. A holomorphic
function $s: \Omega \to \mathcal H$  is called a spanning section for an operator $T$ in the Cowen-Douglas class if  \index{spanning holomorphic section}
$\ker (T-w) s(w) = 0$ and the closed linear span of $\{s(w):w\in \Omega\}$ is $\mathcal H$. Kehe Zhu  in \cite{misra-KZ} proved the existence of a spanning section for an operator $T$ in $B_n(\Omega)$ and showed that it can be used to characterize Cowen-Douglas operators of rank $n$ up to unitary equivalence and similarity. Unfortunately, the existential nature of the spanning section makes it difficult to apply this result in concrete examples. 

First note that the holomorphic frame ${\gamma}_T$ is \emph {not} uniquely determined even if the rank $n=1$. If $\gamma_T$ is any given holomorphic frame for the operator $T$ defined on an open set $\Omega \subseteq \mathbb C$ and $\varphi:\Omega \to \mathbb C$ is a non-vanishing holomorphic function, then $\varphi  {\gamma}_{T}$ is also a holomorphic frame for the line bundle $E_T$. Therefore, a holomorphic frame can't possibly determine the unitary equivalence class of the operator $T$. How does one get rid of this ambiguity in the holomorphic frame to obtain an invariant?  It is evident that 
\begin{equation}\label{curvature}
\mathscr K_T(w) = - \frac{\partial^2}{\partial w \partial\overbar{w}}\log {\|\gamma_T(w)\|^2},\,\, w\in \Omega_0,
\end{equation}
is the same for all holomorphic frames of the form $\varphi \,\gamma_T$, where  $\varphi$ is any non-vanishing holomorphic function on some open set $\Omega_0 \subseteq \Omega$. Since any two holomorphic frames of the operator $T$  must differ by such a holomorphic change of frame, we conclude that $\mathscr K_T$ is a unitary invariant for the operator $T$. The converse is also valid and is well-known: The curvature $\mathscr K_T$ of the line bundle $E_T$ is defined by the formula  \eqref{curvature} and is a complete invariant for the line bundle $E_T$. 

To see the usefulness of this invariant, consider the weighted unilateral shift $W_\lambda$  determined by the weight sequence $\{\sqrt{\tfrac{n+1}{n+\lambda}}\},$ $\lambda> 0$, acting on the Hilbert space $\ell^2$ of square summable sequences. Clearly, the adjoint $W_\lambda^*$ admits a holomorphic frame. For instance, one may choose $\gamma_{W_\lambda^*}(w)=(1, c_1w, c_2w^2, \ldots )$, $w\in \mathbb D$, where $c^2_k$ is the co-efficient of $x^k$ in the binomial expansion of the function $(1-x)^{-\lambda}$. 
It then follows that $\|\gamma_{W_\lambda^*}(w)\|^2 = (1-|w|^2)^{-\lambda}$ and that $\mathscr K_{W_\lambda^*}(w) =-\lambda (1-|w|^2)^{-2}$, $w\in \mathbb D$.  Consequently, using the Cowen-Douglas theorem, we conclude that none of the operators $W_\lambda$ are unitarily equivalent among themselves.

Finding similarity invariants for operators in the class $B_n(\Omega)$ has been somewhat difficult from the beginning. The conjecture made by Cowen and Douglas in \cite{misra-CD} was shown to be false \cite{misra-CM1, misra-CM2}. However, significant progress on the question of similarity has been made recently (cf. \cite{misra-CJ, misra-JJ, misra-DKS, 
misra-jw}).

After a model and a complete set of unitary invariants were provided for the operators in the class $B_n(\Omega)$ in \cite{misra-CD}, it was only natural to follow it up with the  study of a commuting tuple of operators with similar properties. This was started in \cite{misra-CD1} and followed up in the papers \cite{misra-CD2} and \cite{misra-CS}. The approach in the papers \cite{misra-CD2} and \cite{misra-CS} are quite different. We recall below the definition of the class $B_n(\Omega),$ $\Omega \subset \mathbb C^m$, from the paper \cite{misra-CS}. This definition also appears in \cite{misra-CD1} and is implicit in \cite{misra-CD2}.  

Let $\boldsymbol{T}= (T_1, \ldots , T_m)$ be an $m$-tuple of commuting bounded linear operators on a separable complex Hilbert space $\mathcal H$. For $w = (w_1, \ldots, w_m)$ in $\Omega$, let $T-w$ denote the $m$-tuple
$(T_1-w_1, \ldots ,T_m - w_m)$.  Define the operator $D_{\boldsymbol{T}}:\mathcal H\to \mathcal H\oplus  \cdots \oplus \mathcal H$ by 
$$D_{\boldsymbol T}(x) = (T_1 x, \ldots, T_m x),\,\, x\in {\mathcal H}.$$   

\begin{definition}
For a connected, bounded and open subset $\Omega$ of $\mathbb C^m$, a $m$-tuple $\boldsymbol{T}$ is said to be in the Cowen-Douglas class $B_n(\Omega)$ of rank $n$, $n\in \mathbb N$, if 
\begin{enumerate}
\item $\mbox{\rm ran}~D_{\boldsymbol{T - w}}$ is closed for all $w\in \Omega$ 
\item $\mbox{\rm span}~\{ \ker D_{\boldsymbol{T - w}}: w \in \Omega\}$ is dense in
${\mathcal H}$
\item $\dim\, \ker\, D_{\boldsymbol{T - w}}= n$  for all $w\in \Omega$.
\end{enumerate}
\end{definition}
For $m=1$, it is shown in \cite[Proposition 1.12]{misra-CD} that if  $T$ is in $B_n(\Omega),$ then there exists a choice of $n$ eigenvectors in $\ker(T-w),$ which are holomorphic as  functions of $w\in\Omega$ making 
$$E_T:=\{(w,x): w\in\Omega,\, x \in \ker (T-w)\}\subseteq \Omega\times\mathcal H,$$ 
$\pi:E_T\to \Omega$,  $\pi(w, x)= w$, is a rank $n$  holomorphic Hermitian
vector bundle over $\Omega.$ Here is one of the main results from \cite{misra-CD}.

\bigskip
{\bf Theorem} (Cowen and Douglas). \emph{The operators $T$ and $\hat{T}$ in $B_n(\Omega)$ are unitarily
equivalent if and only if the corresponding holomorphic Hermitian
vector bundles $E_T$ and $E_{\hat{T}}$ are equivalent on some open subset $\Omega_0$ of~$\Omega$.}

The existence of the vector bundle $E_{\boldsymbol T}$ follows from 
\cite[Theorem 2]{misra-CS}, while \cite[Theorem 3.7]{misra-CS} provides the analogue of the Cowen-Douglas Theorem  for an arbitrary $m$. Finally, a complete set of invariants in this case are given in \cite{misra-CD2}. 

Crucial in any study of such a class is the problem of finding a canonical model and a set of invariants. For normal operators, the spectral theorem provides a model in the form of a multiplication operator and a complete set of invariants is given by the spectrum, the spectral measure and the multiplicity function. 
Similarly, the Sz.-Nagy -- Foias theory provides a model for a pure completely nonunitary contraction and the characteristic function serves as a complete invariant. Now, we describe a model for the operators, resp. commuting tuples, 
in the Cowen-Douglas class. \index{spectral measure}\index{multiplicity function}
\index{Sz.-Nagy-Foias theory}\index{commuting tuples}

Let $V$ be a $n$-dimensional Hilbert space and ${\mathcal L}(V)$ denote
the vector space of all linear transformations on $V$.  A function
$K:\Omega\times\Omega \to {\mathcal L}(V)$, satisfying
\begin{equation} \label{existence reprod}
\sum_{i,j=1}^N \langle {K(w_i,w_j)\zeta_j}, {\zeta_i}\rangle_V ~\geq~0 ,~~w_1,
\ldots,w_N\in \Omega, ~~\zeta_1,\ldots,\zeta_N \in V, N\geq 1
\end{equation}\index{non negative definite kernel}\index{nnd kernel}
is said to be a {\em non negative definite (nnd) kernel} on $\Omega$.
Given such an nnd kernel $K$ on $\Omega$, it is easy to construct a Hilbert space
${\mathcal H}$ of functions on $\Omega$ taking values in $V$ with the
property
\begin{equation} \label{reproducing property} 
\langle{f(w)},{\zeta}\rangle_V = \langle{f},{K(\cdot,w)\zeta}\rangle_{\mathcal H},~ w\in
\Omega,~\zeta\in V,~f\in {\mathcal H}.
\end{equation}
The Hilbert space ${\mathcal H}$ is simply the completion of the linear span of all
vectors of the form $K(\cdot,w)\zeta$, $w\in \Omega$, $\zeta\in
V$, with inner product defined by (\ref{reproducing property}).

Conversely, let ${\mathcal H}$ be any Hilbert space of functions on
$\Omega$ taking values in $V$.  Let $e_w : {\mathcal H} \to V$ be
the evaluation functional defined by $e_w(f) = f(w)$, $w\in
\Omega$, $f\in {\mathcal H}$.  If $e_w$ is bounded for each $w
\in \Omega$ then it is easy to verify that the Hilbert space ${\mathcal H}$
possesses a reproducing kernel $K(z,w)= e_z e_w^*$, that is,
$K(\cdot ,w)\zeta \in {\mathcal H}$ for each $w\in \Omega$ and $K$
has the reproducing property (\ref{reproducing property}). Finally, 
the reproducing property (\ref{reproducing property}) determines the kernel $K$ uniquely.  We let $(\mathcal H,K)$ be the Hilbert space $\mathcal H$ equipped with the reproducing kernel $K$. \index{reproducing kernel}
\begin{remark} \label{bdder}  Let $ K :\Omega\times \Omega\to \mathcal M_k(\mathbb C)$ be a non-negative definite kernel. For every $\boldsymbol{i}\in \mathbb Z_{+}^m$, $\eta \in \mathbb{C}^k $ and $ w\in\Omega,$ we have  
\begin{itemize} 
\item[\rm (i)] ${\bar{\partial}}^{\boldsymbol{i}}K(\cdot,w)\eta$ is in $(\mathcal H, K),$ 
\item[\rm (ii)]$
\left \langle f,\bar{\partial}^{\boldsymbol{i}} K(\cdot,w)\eta \right \rangle =\left \langle (\partial^{\boldsymbol i}f)(w),\eta\right \rangle_{\mathbb C^k}, 
f\in (\mathcal H, K).$
\end{itemize}
The proof follows from the uniform boundedness principle 
\cite[Proposition 2.1.3]{misra-SG}.
\end{remark}\index{uniform boundedness principle}

Given any $m$-tuple of operators $\boldsymbol{T}$ in $B_n(\Omega)$,  there exists an open subset $U$ of $\Omega$ and  $n$ linearly
independent vectors $\gamma_1(w), \ldots, \gamma_n(w)$ 
in $\ker D_{\boldsymbol{T-w}},$ $w \in U$, such that each of the maps $w \mapsto \gamma_i(w)$ is holomorphic on $U$, see 
\cite[Proposition 1.11]{misra-CD} and \cite[Theorem 2.2]{misra-CS}. Define 
$\hat{\Gamma}: U \to {\mathcal L}(\mathbb{C}^n,{\mathcal H})$ by setting 
$$\hat{\Gamma}(w)\zeta = \sum_{i=1}^n \zeta_i\gamma_i(w),\,\, \zeta=
(\zeta_1, \ldots, \zeta_n)\in \mathbb{C}^n.$$  

Let ${\mathcal O}(U, \mathbb{C}^n)$ denote the
linear space of holomorphic functions on $U$ taking values in $\mathbb C^n$.
Set $U^* : = \{w: \overbar{w} \in U\}$. Define $\Gamma: {\mathcal H}\to {\mathcal O}(U^*, \mathbb{C}^n)$ by
\begin{equation} \label{general construction}
(\Gamma x)({w}) = \hat{\Gamma}(\overbar{w})^* x, ~~x\in {\mathcal H},~w\in U^*.
\end{equation}
Define a sesqui-linear form on ${\mathcal H}_\Gamma = \mbox{ran}~\Gamma$ by 
$\langle {\Gamma} f, \Gamma g\rangle_{\Gamma} = \langle{f},{g}\rangle$, $f,g\in{\mathcal H}$.
The map $\Gamma$ is linear and injective.  Hence ${\mathcal
H}_\Gamma$ is a Hilbert space of $\mathbb{C}^n$-valued holomorphic functions on $U^*$ 
with inner product $\langle {\cdot},{\cdot}\rangle_{{\Gamma}}$ and $\Gamma$ is unitary. 
Then it is easy to verify the following (cf. \cite[pp. 194 ]{misra-CD} and 
\cite[Remarks 2.6]{misra-CS}).
\begin{description}
\item[a)] $K(z,w) = \hat{\Gamma}(\overbar{z})^* \hat{\Gamma} (\overbar{w})$,
$z,w \in U^*$ is the reproducing kernel  for the Hilbert space ${\mathcal H}_{\Gamma}$.
\item[b)] $M_i^* \Gamma = \Gamma T_i$, where $(M_i f)(z) = z_i f(z)$,
$z\in U^*$.
\end{description}
Thus any  commuting $m$-tuple $\boldsymbol T$ of operators in the class $B_n(\Omega)$ may be realized as the adjoint of the $m$-tuple $\boldsymbol M:=(M_1, \ldots ,M_m)$ of multiplication by the coordinate functions on some Hilbert space $\mathcal H$ of holomorphic functions defined on $U^*$ possessing a reproducing kernel $K$. In this representation, clearly, $\Gamma( \gamma_i(w) ) = K(\cdot, \overbar{w}) \varepsilon_i$, $1 \leq i \leq n$ is a holomorphic frame.

We give this correspondence for commuting tuple of operators in $B_1(\Omega)$ adding that except for slight increase in the notational complexity, the same proof works in general. 


Let $\gamma$ be a non-zero
holomorphic section defined on some open subset $U$ of $\Omega$ for the operator $T$ acting on the Hilbert
space $\mathcal H$. Consider the map $\Gamma:\mathcal H \rightarrow
\mathcal O(U^*)$
defined by $\Gamma(x)(z) =
\langle x,\gamma(\bar z) \rangle,\, z\in U^*$. Transplant the
inner product from $\mathcal H$ on the range of $\Gamma$.  The map
$\Gamma$  is now  unitary  from $\mathcal H$ onto ${\rm ran}\,\Gamma$. Define $K$ to be the function $K(z,w) =
\Gamma\big (\gamma(\bar{w})\big )(z) = \langle \gamma(\bar{w}) ,
\gamma(\bar{z}) \rangle$, $z,w\in U^*$. Set $K_w(\cdot):=
K(\cdot,w)$. Thus $K_w$ is the function $\Gamma \big
(\gamma(\bar{w})\big )$. It is then easily verified that $K$ has the
reproducing property, that is,
\begin{eqnarray*}
\langle \Gamma(x)(z), K(z,w) \rangle_{{\rm ran}\,\Gamma} &=&
\langle \big (\langle x, \gamma(\bar{z}) \rangle \big ) ,\big (\langle \gamma(\bar{w}) ,
\gamma(\bar{z}) \big ) \rangle_{{\rm ran}\, \Gamma}\\
&=& \langle \Gamma x, \Gamma(\gamma(\bar{w})) \rangle_{{\rm ran}\,
\Gamma} = \langle x , \gamma(\bar{w}) \rangle_{\mathcal H} \\&=&
\Gamma(x)(w),\,\, x\in \mathcal H,\,\, w\in U^*.
\end{eqnarray*}
It follows that $\|K_w(\cdot)\|^2 = K(w,w)$, $w \in U^*$. Also,
$K_w(\cdot)$ is an eigenvector for the operator $\Gamma \,T_i\,\Gamma^*$, $1\leq i \leq m$, 
with eigenvalue $\overbar{w}_i$:
\begin{eqnarray*}
\Gamma \,T_i\, \Gamma^*( K_w(\cdot)) &=& \Gamma\, T_i\, \Gamma^* \big (
\Gamma(\gamma(\bar{w}))\big )\\
&=&\Gamma\, T_i \,\gamma(\bar{w})\\ &=& \Gamma\, \overbar{w}_i\, \gamma(\overbar{w})\\
&=& \overbar{w}_i\, K_w(\cdot), \,\, w\in U^*.
\end{eqnarray*}
Since the  linear span of the vectors $\{K_w :
w\in U^*\}$ is dense in $(\mathcal H, K)$  (see \cite[Corollary 1.13]{misra-CD}), it follows that $\Gamma \,T_i\,
\Gamma^*$ is the adjoint $M_i^*$ of the multiplication operator $M_i$ acting on $(\mathcal H, K)$. We therefore assume, without loss
of generality, that an operator $\boldsymbol T$ in $B_1(\Omega)$ has been realized as the
adjoint $\boldsymbol M^*$ of the multiplication operator $\boldsymbol M$ on some Hilbert space $(\mathcal H, K)$ of holomorphic functions on $U^*$ possessing a reproducing kernel $K$.

Moreover, starting from any nnd kernel $K$ defined on $\Omega$ taking values in $\mathcal M_n(\mathbb C)$ and fixing a $w_0$ in $\Omega$, we note that the function
$$ K_0(z,w) = K(w_0,w_0)^{\frac{1}{2}}\varphi(z)^{-1}K(z,w)\overbar{\varphi(w)^{-1}}
K(w_0,w_0)^{\frac{1}{2}}
$$
is defined on some open neighbourhood $U$ of $w_0$ on which $\varphi(z) = K(z,w_0)$ is holomorphic and non-zero. Thus, the $m$ - tuple $\boldsymbol M$ defined on $(\mathcal H, K)$ is unitarily equivalent to the the $m$ - tuple $\boldsymbol M$ on $(\mathcal H_0,K_0)$, see \cite{misra-CD, misra-CS}. 

The kernel $K_0$ is said to be normalized at $w_0$ in the sense that $K_0(z,w_0)=I_n$ for each $z\in U$.


The commuting $m$-tuple of multiplication operators acting on the Hilbert space $\mathcal H_\Gamma$ is called the canonical model. This terminology is justified by \cite[Theorem 4.12(a)]{misra-CS}, it says, ``the canonical models associated with two generalized Bergman kernels are unitarily equivalent if and only if the normalized forms of the kernels are unitarily equivalent via a unitary that does not depend on points of $\Omega$.'' 
\index{canonical model}
 
It is possible to impose conditions on a kernel function
$K:\Omega\times \Omega \to \mathbb{C}$ so that each of the multiplication operators $M_1, \ldots, M_m$ are bounded on the Hilbert space $(\mathcal H, K)$. Additional conditions, explicitly given in \cite{misra-CS}, on $K$ ensure that $\boldsymbol{M}^*:= (M_1^*, \ldots ,M_m^*)$ is in
$B_1(\Omega^*)$. If we set the curvature $\mathscr K$ of the $m$-tuple $\boldsymbol M^*$ to be the $(1,1)$ - form $$\mathscr K(w):= - \sum_{i,j=1}^m \mathscr K_{i,j}(w)dw_i \wedge d\overbar{w}_j,$$
where $\mathscr K_{i,j}(\bar{w}) =  \big ( \frac{\partial^2}{\partial w_i \partial \overbar{w}_j} \log K \big ) (w,w),$
then the unitary equivalence class of operators $\boldsymbol T$ in $B_1(\Omega^*)$, which we assume is of the form $\boldsymbol M^*$ on some reproducing kernel Hilbert space $(\mathcal H, K)$, is determined by the curvature $(1,1)$ form. 

In the case of a commuting $m$-tuple of operators $\boldsymbol T$ in the Cowen-Douglas class $B_n(\Omega)$, the existence of a spanning section was proved in \cite{misra-ES}. Some examples of spanning sections are given in \cite{misra-BGMR}.

\subsection{Curvature inequalities} 
We may assume, without loss
of generality, that an operator $T$ in $B_1(\Omega)$ has been realized as the
adjoint $M^*$ of the multiplication operator $M$ on some Hilbert space $(\mathcal
H, K)$ of holomorphic functions on $\Omega^*$ possessing a reproducing kernel
$K:\Omega^* \times \Omega^* \to \mathbb C$. For the unit disc $\mathbb D$, the distinction between $\mathbb D$ and $\mathbb D^*$ disappears and we write $K(z,w)$, when strictly speaking, we should be writing $K(\bar{z}, \bar{w})$, $z,w \in \mathbb D$. 
The curvature of the operator $M^*$ may be also written in the form
\begin{eqnarray*}\label{pos1} 
\mathscr K(w) &=& -\frac{\|\gamma(w)\|^2\|\gamma^{\,\prime}(w)\|^2 - |\langle
\gamma^{\,\prime}(w), \gamma(w)\rangle|^2}{\|\gamma(w)\|^4}\end{eqnarray*}
for some holomorphic frame $\gamma$. In particular, choosing $\gamma(\overbar{w}) = K(\cdot,{w}),\, w\in \mathbb D,$ we also have 
\begin{eqnarray*}
\mathscr K(\bar{w}) &=& - \frac{\partial^2}{\partial w \partial\overbar{w}}\log {K(w,w)} = - \frac{K(w,w) (\partial \overbar{\partial} K)(w,w) - |(\partial K)(w,w)|^2 }{K(w,w)^2}.
 \end{eqnarray*}
In either case, since $K$ is nnd, the Cauchy - Schwarz inequality applies, and we see that the numerator is non-negative.
Therefore, $\frac{~\partial^2}{\partial{w}\partial{\bar{w}}}\log
K(w,w)$ must be a non-negative function. 

The contractivity of the adjoint $M^*$ of the multiplication
operator $M$ on some reproducing kernel Hilbert space $(\mathcal
H, K)$ is equivalent to the requirement that ${K}^\ddag(z,w):=(1 -
z\bar{w})K(z,w)$ is nnd on $\mathbb D$. This is easy to prove as long as $K$ is positive definite. However, with a little more care, one can show this assuming only  that $K$ is nnd, see \cite[Lemma 2.1.10]{misra-SG}.  

Now, let $T$ be any contraction in $B_1(\mathbb D)$ realized in the form of 
the adjoint $M^*$ of the multiplication operator $M$ on some reproducing kernel Hilbert space $(\mathcal H, K)$.   Then 
we have \begin{eqnarray*} \frac{~\partial^2}{\partial{w}\partial{\bar{w}}}
\log K(w,w) = \frac{~\partial^2}{\partial{w}\partial{\bar{w}}}
\log {\frac{1}{(1 - |w|^2)}} +
\frac{~\partial^2}{\partial{w}\partial{\bar{w}}}\log K^\ddag(w,w),
\, w\in \mathbb D.\end{eqnarray*} 
Let $S$ be the unilateral shift acting on $\ell^2$. Choosing a holomorphic frame $\gamma_{S^*},$ say $\gamma_{S^*}(w)=(1,w,w^2, \ldots )$, it follows that $\|\gamma_{S^*}(w)\|^2 = (1-|w|^2)^{-1}$ and that $\mathcal K_{S^*}(w) =-(1-|w|^2)^{-2}$, $w\in \mathbb D$. 
We can therefore, rewrite the previous equality in the form 
\begin{eqnarray*} \mathscr
K_{M^*}(w) - \mathcal K_{S^*}(w) = -
\frac{~\partial^2}{\partial{w}\partial{\bar{w}}}\log K^\ddag(w,w) \leq 0,\,\,
w\in\mathbb D. \end{eqnarray*} 
In consequence, we have
$$
\mathscr K_{M^*}(w) \leq \mathcal K_{S^*}(w), \,\,w\in \mathbb D.
$$
Thus the the operator $S^*$ is an extremal operator in the class of all contractive Cowen-Douglas operator in $B_1(\mathbb D)$. The extremal property of the operator $S^*$ prompts the following question due to R. G. Douglas.  
\index{contractive Cowen-Douglas operator}\index{contraction}

\emph{A question of R. G. Douglas.} For a contraction $T$ in $B_1(\mathbb D),$ if $\mathscr K_T(w_0) =- (1-|w_0|^2)^{-2}$ for some fixed $w_0$ in $\mathbb D,$ then does it follow that $T$ must be unitarily equivalent to the operator $S^*$?

It is known that the answer is negative, in general,  however it has an affirmative answer if, for instance, $T$ is a homogeneous contraction in $B_1(\mathbb D),$ see 
\cite{misra-GMshift}. From the simple observation that $\mathcal{K}_T(\bar{\zeta})  = - (1-|\zeta|^2)^{-2}$ for some $\zeta\in \mathbb D$ if and only if the two vectors ${K}^\ddag_{\zeta}$ and $\bar{\partial}{K}^\ddag_{\zeta}$ are linearly dependent,  it follows that the question of Douglas has an affirmative answer in the class of contractive, co-hyponormal backward weighted shifts.  The  Question of Douglas for all those operators $T$ in $B_1(\mathbb D)$   possessing two additional properties, namely, $T^*$ is $2$ hyper-contractive and $(\phi(T))^*$ has the wandering subspace property for any  bi-holomorphic automorphism $\phi$ of $\mathbb D$ mapping $\zeta$ to $0.$ This is Theorem 3.6 of the  of the paper  \cite{misra-MR}.\index{co-hyponormal backward shift}\index{hypercontractive operator}
\index{wandering subspace}

Now suppose that the domain $\Omega$ is not simply connected. In this case, replacing the contractivity of the operator $T$  by the contractivity of the homomorphism $\varrho_T$ induced by an operator $T$, namely, $\varrho_T(r) = r(T)$,  $r\in {\rm Rat}(\Omega^*)$, the algebra of rational functions with poles off $\overbar{\Omega}^*$, we assume that  
$\|r(T)\| \leq \|r\|_{\Omega^*, \infty}.$ For such operators $T,$ the curvature inequality 
\begin{align*}
\mathscr{K}_{T} (\bar{w}) &\leq -4\pi ^2(S_{\Omega^*}(\bar{w},\bar{w}))^2,\;\;\;\bar{w}\in\Omega^*,
\end{align*}
where $S_{\Omega^*}$ is the S\"{z}ego kernel of the domain $\Omega^*,$ was established in \cite{misra-GM}.
Equivalently, since $S_{\Omega}(z,w) = S_{\Omega^*}(\bar{w},\bar{z}),\,\;z,w\in \Omega,$ the curvature inequality takes the form 
\begin{align}\label{eq: alternative CI}
\frac{\partial ^2}{\partial w \partial \bar{w}}\mbox{log} K(w,w)\geq 4\pi ^2 (S_{\Omega}(w,w))^2, \;\;\;w \in \Omega.
\end{align}\index{curvature inequality}
The curvature inequality in \eqref{eq: alternative CI} is for operators $T$ in $B_1(\Omega ^*)$ for which $\overbar{\Omega} ^*$ is a spectral set. 
It is not known if there exists an extremal operator $T$ in $B_1(\Omega^*)$, that is, if  
$\mathscr K_T(w)= - 4\pi ^2 (S_{\Omega^*}(w,w))^2,\,\,w\in\Omega^*,$ for some  operator $T$ in  $B_1(\Omega^*)$.  Indeed, from a result of Suita (cf. \cite{misra-SUITA}), it follows that the adjoint of the multiplication operator on the Hardy space $(H^2(\Omega),ds)$ is not extremal. It was shown in \cite{misra-GM} that for any fixed but arbitrary $w_0\in \Omega,$ there exists an operator $T$ in $B_1(\Omega^*)$  for which equality is achieved, at  $w=w_0,$ in the inequality \eqref{eq: alternative CI}. The question of Douglas is the question of  uniqueness of such an operator. It was partially answered recently in \cite{misra-Ramiz}. The precise result is that these ``point-wise'' extremal operators are determined uniquely within the class of the adjoint of the bundle shifts introduced in \cite{misra-AD}. It was also shown in the same paper that each of these bundle shifts can be realized as  a multiplication operator on a Hilbert space of weighted Hardy space and conversely.  Some very interesting inequalities involving, what the authors call ``higher order curvature'', are given in \cite{misra-WZ}. \index{bundle shift}\index{higher order curvature}

\subsection{Homogeneous operators} \index{homogeneous operator}
The question of Douglas discussed before has an affirmative answer in the class of homogeneous operators in the Cowen-Douglas class. An operator $T$ with its spectrum $\sigma(T)$ contained in the closed unit disc  $\overbar{\mathbb D}$ is said to be homogeneous if $U_\varphi^* T U_\varphi = \varphi(T)$ for each bi-holomorphic automorphism $\varphi$ of the unit disc and some unitary $U_\varphi$. It is then natural to ask what are all the homogeneous operators. Let us  describe (see 
\cite{misra-GMshift, misra-DW, misra-BisMis}) the homogeneous operators in $B_1(\mathbb D)$. 
We first show that the equivalence class of a holomorphic Hermitian line bundle $L$ defined on a bounded planar domain $\Omega$ is determined by its curvature $\mathscr K_L.$ 

\begin{proposition}\label{curvinv}
Suppose that  $E$ and $F$ are two holomorphic Hermitian line bundles defined on some  bounded domain $\Omega\subseteq \mathbb C^m$. Then they are locally equivalent as holomorphic Hermitian bundles if and only if $\mathscr K_E = \mathscr K_F$. 
\end{proposition}\index{Hermitian metric}

\begin{proof} For simplicity, first consider the case of $m=1$. 
Suppose that $E$ is a holomorphic line bundle over the domain
$\Omega\subseteq \mathbb C$ with a hermitian metric $G(w) =
\langle {\gamma_w}, {\gamma_w}\rangle$, where $\gamma$ is a holomorphic
frame. The curvature $\mathscr K_E$ is given by the formula $\mathscr K_E(w) = - 
(\frac{\partial^2}{\partial w \partial\bar{w}}\log {G})(w)$, for $w\in \Omega$. 
Clearly, in this case, $\mathscr K(w)\equiv 0$ on $\Omega$ is the same as saying that
$\log G$ is harmonic on $\Omega$.  Let $F$ be a second line bundle
over the same domain $\Omega$ with the metric $H$ with respect to
a holomorphic frame $\eta$.  Suppose that the two curvatures
$\mathscr K_E$ and $\mathscr K_F$ are equal.  It then follows that $u=\log
(G/H)$ is harmonic on $\Omega$ and thus there exists a harmonic
conjugate $v$ of $u$ on any simply connected open subset
$\Omega_0$ of $\Omega$.  For $w\in\Omega_0$, define
$\tilde{\eta}_w = e^{(u(w)+iv(w))/2} \eta_w$. 
Then clearly,
$\tilde{\eta}_w$ is a new holomorphic frame for $F$, which we can use
without loss of generality. Consequently, we have the metric $H(w) =
\langle {\tilde{\eta}_w}, {\tilde{\eta}_w} \rangle$ relative to the frame $\tilde{\eta}$ for the vector bundle $F$. We have that
\begin{eqnarray*}
H(w) &=& \langle{\tilde{\eta}_w},{\tilde{\eta}_w}\rangle\\
&=& \langle{e^{(u(w)+iv(w))/2}{\eta}_w}, {e^{(u(w)+iv(w))/2}{\eta}_w} \rangle\\
&=& e^{u(w)}\langle{{\eta}_w},{{\eta}_w}\rangle\\
&=& G(w).
\end{eqnarray*}
This calculation shows that the map $\tilde{\eta}_w \mapsto \gamma_w$
defines an isometric holomorphic bundle map between the vector bundles $E$ and $F$. 

To complete the proof in the general case, recall that 
$$\sum_{i,j=1}^m \big ( \frac{\partial^2}{\partial w_i \partial \bar{w}_j} 
\log G/H \big ) (w) dw_i \wedge d\bar{w}_j=0$$ 
means that the function is $u:=H/G$ is pluriharmonic. The proof then follows exactly the same way as in the case of $m=1$. Indeed, as in  \cite[Theorem 1]{misra-CD1}, the map 
\begin{equation}\label{CD unitary}
U\Big ( \sum_{|I| \leq n} \alpha_I (\bar{\partial}^I
\tilde{\eta})({w_0})\Big ) = \sum_{|I| \leq n} \alpha_I (\bar{\partial}^I
\gamma)(w_0), \,\, \alpha_I \in \mathbb C,
\end{equation}
where $w_0$ is a fixed point in $\Omega$ and $I$ is a multi-index
of length $n$, is well-defined, extends to a unitary operator on
the Hilbert space spanned by the vectors $(\bar{\partial}^I
\tilde{\eta})({w_0})$ and intertwines  the two $m$-tuples of operators in
$\mathrm B_1(\Omega)$ corresponding to the vector bundles $E$ and
$F$.
\end{proof}

As shown in \cite{misra-CD}, it now follows that the curvature $\mathcal K_T$ of an operator $T$ in $B_1(\Omega)$ determines the unitary equivalence class of $T$ and conversely.

\bigskip
\noindent{\bf Theorem (Cowen-Douglas)} \emph{Two operators $T$ and $\tilde{T}$ belonging to the class $B_1(\Omega)$ are unitarily equivalent if and only if $\mathscr K_T = \mathscr K_{\tilde{T}}$. }

\begin{proof}
Let $\gamma_T(w)$ be a holomorphic frame for the line bundle $E_T$ over $\Omega$ corresponding to an operator $T$ in $\mathrm B_1(\Omega)$.
Thus the real analytic function $G_T(w):= \langle{\gamma_T( w)},{\gamma_T(w)}\rangle$ is the Hermitian metric for the bundle $E_T$. Similarly, let $\gamma_{\tilde{T}}$ and $G_{\tilde{T}}$ be the holomorphic frame and the Hermitian metric corresponding to the operator $\tilde{T}$. If $T$ and $\tilde{T}$ are unitarily equivalent, then the eigenvector $\gamma_{\tilde{T}}(w)$ must be a multiple, say $c$ depending on $w$, of the eigenvector $\gamma_T(w)$. However, since both $\gamma_{\tilde{T}}$ and $\gamma_T$ are holomorphic, it follows that $c$ must be holomorphic. Hence $G_{\tilde{T}} (w) = |c(w)|^2  G_T(w)$ and we see that $\mathscr K_T = \mathscr K_{\tilde{T}}$. 

Conversely, if the two curvatures are equal, from Proposition \ref{curvinv}, we find that we may choose, without loss of generality, a holomorphic frame $\gamma_{\tilde{T}}$ for the operator $\tilde{T}$ such that $G_{\tilde{T}}= G_T$.  Since the linear span of the vectors $\gamma_T(w)$ and $\gamma_{\tilde{T}}(w)$ are dense, it follows that the map $U$ taking  $\gamma_T(w)$ to $\gamma_{\tilde{T}}(w)$ is isometric. Extending it linearly, we obtain a unitary operator that intertwines $T$ and $\tilde{T}$ 
\end{proof}

We now explain how the curvature can be extracted directly from an operator $T:\mathcal H \to \mathcal H$ which is in  the class $B_1(\Omega)$. Let  
$\gamma$ be a holomorphic frame for the operator $T$.  Recall that $\gamma^\prime(w),$ $w\in \Omega$, is also in the Hilbert space $\mathcal H$. The  restriction $N(w)$ of the operator $T-w I$ to the two dimensional subspaces $\{\gamma(w), \gamma^\prime(w)\},\,\, w\in \Omega$ is nilpotent and encodes important information about the operator $T$. 

With a little more effort, one may work with commuting tuples of bounded operators on a Hilbert space possessing an open set $\Omega \subseteq \mathbb C^m$ of joint eigenvalues. We postpone the details to Section \ref{locop}.

Let $\mathcal N(w) \subseteq \mathcal H$, $w\in \Omega$, be the subspace consisting of the two linearly independent vectors $\gamma(w)$ and $\gamma^\prime(w)$.
There is a natural nilpotent action $N(w):=(T - wI)_{| \mathcal N(w)}$ on the space $\mathcal N(w)$ determined by the rule 
$$
\gamma^\prime(w) \stackrel{N(\!w\!)}{\longrightarrow} \gamma(w) \stackrel{N(\!w\!)}{\longrightarrow} 0.
$$

Let $e_0(w),\, e_1(w)$ be the orthonormal basis for $\mathcal N(w)$ obtained from $\gamma(w),\,\gamma^\prime(w)$\index{Gram-Schmidt orthonormalization}
 by the Gram-Schmidt orthonormalization. The matrix representation of $N(w)$ with respect to this orthonormal 
basis is of the form $\Big (\begin{smallmatrix} 0& h(w) \\ 0&0\end{smallmatrix}\Big )$. It is easy to compute $h(w)$.  Indeed, we have 
$$h(w) = \frac{\|\gamma(w)\|^2}{(\|\gamma^\prime(w)\|^2 \|\gamma(w)\|^2 - |\langle \gamma^\prime(w), 
\gamma(w)\rangle|^2)^{\frac{1}{2}}}.$$
We observe that $\mathscr K_T(w) = - h(w)^{-2}.$

\index{automorphism}\index{functional calculus}
Let $\varphi$ be a bi-holomorphic automorphism of the unit disc $\mathbb D$. Thus $\varphi(z)$ is of the form $e^{i\theta} \frac{z-\alpha}{1-\overbar{\alpha}z}$ for some $\theta, \alpha$,   $0\leq \theta < 2\pi,\, \alpha \in \mathbb D$. 
Since the spectrum of  $T$ is contained in $\overbar{\mathbb D}$ and $q(z) = 1-\overbar{\alpha}z$ does not vanish on it, we can define $\varphi(T)$ to be the operator $p(T) q(T)^{-1}$, where $p(z) = z-\alpha$. This definition coincides with the usual holomorphic functional calculus. It is not hard to prove that $\varphi(T)$ is in $B_1(\mathbb D)$, whenever $T$ is in $B_1(\mathbb D)$, see \cite{misra-MR}. 

\begin{theorem}
Let $T$ be an operator in $\mathrm B_1(\mathbb D)$. Suppose that $\varphi(T)$ unitarily equivalent to $T$ for each bi-holomorphic automorphism of $\varphi$ of $\mathbb D$.  Then 
$$
\mathscr K_T(\alpha) = (1 - |\alpha|^2)^{-2} \mathscr K_T(0),
$$
where $-\lambda = \mathcal K_T(0) < 0$ is arbitrary.
\end{theorem}

\begin{proof}
For each fixed but arbitrary $w\in \mathbb D$, we have  
$$\varphi(T)_{|\ker (\varphi(T) - \varphi(w))^2} = \varphi\big (T_{|\ker(T-w)^2}\big ).$$ 
Since $T_{|\ker(T-w)^2}$ is of the form  $\left (\begin{smallmatrix}
             w & h_T(w) \\
		0 & w
            \end{smallmatrix} \right )$ and 
$$\varphi \left (\begin{smallmatrix}
             w & h_T(w) \\
		0 & w
            \end{smallmatrix} \right ) = \left (\begin{smallmatrix} 
     \varphi(w) & \varphi^\prime(w) h_T(w)\\
		0 & \varphi(w)
   \end{smallmatrix} \right ),$$ 
it follows that
\begin{eqnarray*}
\varphi(T)_{|\ker (\varphi(T) - \varphi(w))^2} 
= \begin{pmatrix}
     \varphi(w) & \varphi^\prime(w) h_T(w)\\
		0 & \varphi(w)
    \end{pmatrix}
&\cong& \begin{pmatrix}
     \varphi(w) & |\varphi^\prime(w)| h_T(w)\\
		0 & \varphi(w)
    \end{pmatrix},\\
\end{eqnarray*}
where we have used the symbol $\cong$ for unitary equivalence. Finally, we have 
\begin{eqnarray*}
\big ( - \mathscr K_{\varphi(T)}(\varphi(w)) \big )^{-1/2}&=&h_{\varphi(T)}(\varphi(w))\\
&=& |\varphi^\prime(w)| h_T(w) \\
&=&  |\varphi^\prime(w)|\big (- \mathscr K_T(w)\big )^{-1/2}.
\end{eqnarray*}
This is really a ``change of variable formula for the curvature'', which can be obtained directly using the chain rule.

Put $w=0$, choose $\varphi = \varphi_\alpha$ such that $\varphi(0) = \alpha$.
In particualr, take $\varphi_\alpha(z) = \frac{\alpha - z}{1 - \bar{\alpha} z}$. Then
\begin{eqnarray} \label{transform}
\mathscr K_{{\varphi_\alpha}(T)}(\alpha) &=& \mathscr K_{{\varphi_\alpha}(T)}(\varphi_\alpha(0))= |\varphi^\prime(0)|^{-2} \mathscr K_T(0) \nonumber\\
&=& (1 - |\alpha|^2)^{-2} \mathscr K_T(0),\quad \alpha \in \mathbb D.
\end{eqnarray}

Now, suppose that $\varphi_\alpha(T)$ is unitarily equivalent to  $T$ for all $\varphi_\alpha$, $\alpha \in \mathbb D$.  Then $K_{{\varphi_\alpha}(T)}(w) = \mathcal K_T(w)$ for all $w\in \mathbb D$. Hence
\begin{equation} \label{homog}
 (1 - |\alpha|^2)^{-2} \mathscr K_T(0) = K_{{\varphi_\alpha}(T)}(\varphi_\alpha(0))
=K_T(\alpha),\quad \alpha \in \mathbb D.
\end{equation}
(Here the first equality is the change of variable formula given in \eqref{transform} and the second equality follows from equality of the curvature of two unitarily equivalent operators.)
\end{proof}

\begin{corollary}
If $T$ is a homogeneous operator in $B_1(\mathbb D)$, then it must be the adjoint of the multiplication operator on the reproducing kernel Hilbert space $\mathcal H^{(\lambda)}$ determined by the reproducing kernel $K^{(\lambda)}(z,w):= (1-z\bar{w})^{-2\lambda}$.
\end{corollary}

\begin{proof}
It follows from the Theorem that if the operator $T$ is homogeneous, then  the corresponding metric $G$ for the bundle $E_T$, which is determined up to the
square of the absolute value of a holomorphic function, is of the form:
$$
G(w) = (1 - |w|^2)^{-2\lambda},\quad w\in \mathbb D,\,\lambda >0.
$$
This corresponds to the reproducing kernel $K$ (obtained via polarization of the real analytic function $G$) of the form:\index{polarization}
$$
K^{(\lambda)}(z, w) =  (1-z \bar{w})^{-2\lambda}, \quad\lambda >0;\,\, z, w \in \mathbb D.
$$
\end{proof}\index{Bergman kernel}
The kernel $B(z, w) = (1-z \bar{w})^{-2}$ is the reproducing kernel of the Hilbert space of square integrable (with respect to area measure) holomorphic functions defined on the unit disc $\mathbb D$ and is known as the Bergman kernel. The kernel $K^{(\lambda)}$ is therefore a power of the Bergman kernel and the Hilbert space $\mathcal H^{(\lambda)}$ determined by $K^{(\lambda)}$ is known as the weighted Bergman space.  The adjoint of the multiplication operator $M$ on the Hilbert space $\mathcal H^{(\lambda)}$ corresponding to the reproducing kernel $K^{(\lambda)}$ is in $B_1(\mathbb D)$. 

If $\lambda > \tfrac{1}{2}$, then the multiplication operator $M^{(\lambda)}$ is subnormal and the inner product in the Hilbert space $\mathcal H^{(\lambda)}$ is induced by  the measure 
$$
d\mu(z) = \tfrac{2\lambda - 1}{\pi}  (1-|z|^2)^{2\lambda - 2} dz d \bar{z}.
$$  
For $\lambda =\tfrac{1}{2}$, this operator is the usual shift on the Hardy space. However, for $\lambda < \tfrac{1}{2}$, there is no such measure and the corresponding operator $M^{(\lambda)}$ is not a contraction, not even power bounded and therefore not subnormal.  

\subsection{Quotient and submodules}\index{quotien submodules}
The interaction of one-variable function theory and functional analysis with operator theory over the past half century has been extremely fruitful. Much of the progress in multivariable
spectral theory during the last two decades was made possible by the use of methods from several complex variables, complex analytic and algebraic geometry.   A unifying approach to many of  these problems is possible in the language of  Hilbert modules.
\index{Hilbert modules}\index{polynomial ring}

For any ring $R$ and an ideal  $I \subseteq R$, the study of the pair $I$ and  $R/I$ as modules over the ring $R$ is natural in algebra.
However, if one were to assume that the ring $R$ has more
structure, for instance, if $R$ is taken to be the polynomial ring $\mathbb C[\boldsymbol z]$, in $m$ - variables, equipped with the supremum norm over some bounded domain $\Omega$ in $\mathbb C^m$, then the study of a pair analogous to $I$ and  $R/I$, as above, still makes sense and is important. Fix an inner product on the algebra $\mathbb C[\boldsymbol z]$. The completion of $\mathbb C[\boldsymbol z]$ with respect to this inner product is a Hilbert space, say  $\mathcal {M}$. It is natural to assume that the natural action of point-wise multiplication module action $\mathbb C[\boldsymbol z] \times \mathbb C[\boldsymbol z] \to \mathbb C[\boldsymbol z]$ extends continuously to $\mathbb C[\boldsymbol z] \times \mathcal{M} \to \mathcal{M}$ making $\mathcal M$ a module over $\mathbb C[\boldsymbol z]$.  Natural examples are the Hardy and Bergman spaces on some bounded domain $\Omega\subseteq \mathbb C^m$. Here the module map is induced by point-wise multiplication, namely,  ${\mathbf m}_p(f) = p f,\, p \in \mathbb C[\boldsymbol z],\, f\in \mathcal M$. A Hilbert module need not be obtained as the completion of the polynomial ring, more generally, a Hilbert module is simply a Hilbert space equipped with an action of a ring $R$.
When this action is continuous in both  the variables, the Hilbert space $\mathcal M$ is said to be a Hilbert module over the polynomial ring $\mathbb C[\boldsymbol z]$.  

A closed subspace $\mathcal S$ of $\mathcal M$ is said to be a submodule of $\mathcal M$ if $\mathbf m_p h\in \mathcal S$ for all $ h\in \mathcal S$ and $p\in \mathbb C[\boldsymbol z]$.  \index{submodule}
The quotient module $\mathcal Q :={\mathcal M}/{\mathcal S}$ is  the Hilbert space $\mathcal S^\perp$, where the module multiplication is defined to be the compression of the module multiplication on $\mathcal M$ to the subspace $\mathcal S^\perp$, that is, the module action on $\mathcal Q$ is given by  $\mathbf m_p (h) = P_{\mathcal S^\perp} (\mathbf m_p h)$, $h\in \mathcal S^\perp$. Two Hilbert modules  $\mathcal M_1$ and $\mathcal M_2$ over $\mathbb C[\boldsymbol z]$ are said to be isomorphic if there exists a unitary operator $U:\mathcal M_1\to \mathcal M_2$ such that $U(p \cdot h)= p \cdot Uh$, $p\in \mathbb C[\boldsymbol z]$, $h\in \mathcal M_1$.  A Hilbert module {$\mathcal M$} over the polynomial ring
{$\mathbb C[\mathbf z]$} is said to be in $B_n(\Omega)$ if
{$\dim\mathcal H/\mathfrak m_w\mathcal M =
n <\infty$} for all $w\in\Omega,$ and $\cap_{w\in \Omega} \mathfrak m_w\mathcal M = \{0\}$, where $\mathfrak m_w$ is the maximal ideal\index{maximal ideal} 
in $\mathbb C[\mathbf z]$ at $w$. \index{analytic Hilbert module}
In practice, it is enough to work with analytic Hilbert modules defined below.

\begin{definition}
A Hilbert module $\mathcal M$ over $\mathbb{C}[\boldsymbol z]$ is called an analytic Hilbert module if the Hilbert space $\mathcal M$ consists of holomorphic functions on some bounded domain $\Omega\subseteq \mathbb{C}^m$,
$\mathbb{C}[\boldsymbol z]\subseteq\mathcal M$ is dense in $\mathcal M$, and
$\mathcal H$ possesses a reproducing kernel  $K:\Omega\times\Omega\to\mathcal L(V)$, where $V$ is some finite dimensional linear space. 

The module action in an analytic Hilbert module $\mathcal M$  is given by point-wise multiplication, that is, ${\mathfrak m}_p(h)(\boldsymbol z) = p(\boldsymbol z) h(\boldsymbol z),\, \boldsymbol z\in \Omega$.
\end{definition}

\index{locally free module}\index{quasi-free module}
There are many closely related notions like the locally free module and the quasi-free module (cf. \cite{misra-ChenRGD, misra-quasi}). No matter which notion one adopts, the goal is to ensure the existence of a holomorphic Hermitian vector bundle such that the equivalence class of the module and that of the vector bundle are in one to one correspondence.  The generalized Bergman kernel and the sharp kernel appearing in \cite{misra-CS} and \cite{misra-AS} achieve a similar objective. The polynomial density in an analytic Hilbert module ensures that the joint kernel 
$$\cap_{i=1}^m \ker(M_i - w_i)^*, \, w\in \Omega^*,$$ 
is of constant dimension, see \cite[Remark, pp. 5]{misra-PIASeqHM}. Let $K$ be the reproducing kernel of the analytic Hilbert module $\mathcal M$ and  $\{\varepsilon_i\}_{i=1}^n,$ be a basis of the linear space $V$. Evidently,  the map 
$$\gamma_i: w\to K(\cdot, \bar{w}) \varepsilon_i,\,\, w\in \Omega^*,\,i=1,\ldots, n,$$ 
serves as a holomorphic frame for the Hilbert module $\mathcal M$. This way, we obtain a holomorphic Hermitian vector bundle $E_{\mathcal M}$.

Let $\mathcal{Z} \subseteq \Omega$ be an analytic submanifold, $T\Omega = T\mathcal{Z}\stackrel{\cdot}{+}N\mathcal{Z}$ be the decomposition of the tangent bundle of $\Omega$.\index{tangent bundle}\index{normal bundle}
Pick a basis for the normal bundle $N\mathcal{Z}$, say   $\partial_i,\,1\leq i \leq p$. 
Fix the submodule $\mathcal M_0 \subseteq \mathcal{M}$ of all
functions $f$ in $\mathcal{M}$ vanishing on  $\mathcal{Z}$ to total order $k$, that is, $\partial^{\boldsymbol{\alpha}}f_{|{\rm res}~\mathcal Z} = 0$ for\index{total order}
all multi index $\boldsymbol{\alpha}$ of length less or equal to $k$.
We now have a a short exact sequence\index{short exact sequence}
$$0 \longleftarrow  {\mathcal{Q}} \longleftarrow
{\mathcal{M}} \stackrel{X}{\longleftarrow} {\mathcal M_0} \longleftarrow 0,$$
where $\mathcal{Q} = \mathcal{M} \ominus \mathcal M_0$ is the quotient module and $X$ is the inclusion map.

One of the fundamental problems is to find a canonical model and 
obtain a (complete) set of unitary invariants for the quotient module $\mathcal{Q}$.

If the submodule is taken to be the maximal set of functions vanishing on an analytic hypersurface  ${\mathcal Z}$ in $\Omega$, then appealing to an earlier result of Aronszajn \cite{misra-A} the following theorem was proved in \cite{misra-rgdgm} to
analyze the quotient module $\mathcal Q$.  Set 
$$\mathcal M_{\text{res}} = \{f_{| \mathcal Z}\mid f\in \mathcal M\}$$ 
and 
$$\|f_{|\mathcal Z}\|_{\text{res}} = \inf\{\|g\|_{\mathcal M}\mid
g\in \mathcal M, g_{|\mathcal Z} = f_{|\mathcal Z}\}.$$

\bigskip
\noindent{\bf Theorem} (Aronszajn).
\emph{The restriction $\mathcal M_{\text{res}}$ is a  Hilbert module over $\mathbb C[\boldsymbol z]$ possessing a reproducing kernel $K_{\text{res}}$, which is the restriction of $K$ to $\mathcal Z$, that is,  $
K_{\text{res}}(\cdot,\boldsymbol{w}) = K(\cdot,\boldsymbol{w})|_{\mathcal Z}$ for $\boldsymbol{w}$ in $\mathcal Z$. Furthermore, 
$\mathcal Q$ is isometrically isomorphic to $\mathcal M_{\text{res}}$.}

\bigskip
As an example, consider the Hardy module $H^2(\mathbb D^2)$.   Since the S\"{z}ego kernel
$$\mathbb S_{\mathbb D^2}(\mathbf z, \mathbf w) = \tfrac{1}{(1-z_1\bar{w}_1)}\tfrac{1}{(1-z_2\bar{w}_2)}$$
is the reproducing kernel of  $H^2(\mathbb D^2)$, restricting it to the hyper-surface  $z_1-z_2=0$ and using new coordinates $u_1=\tfrac{z_1-z_2}{2}$, $u_2=\tfrac{z_1+z_2}{2}$, we see that $
K_{\mathcal Q}(u,v) = \frac1{(1-u\overbar{v})^2}$ for $u,v$ in $\{z_1-z_2=0\}$. This is a
multiple of the kernel function for the Bergman space $L_{\rm hol}^2(\mathbb D)$. Hence the quotient module is isometrically isomorphism to the Bergman module since
multiplication by  a constant doesn't change the isomorphism class of a 
Hilbert module.\index{Bergman module}\index{Hardy module}

Thus the extension of Aronszajn's result provides a model for the quotient module. However, Hilbert modules determined by different kernel functions may be equivalent. To obtain invariants one approach
is to appeal to the inherent complex geometry.  Assume that the 
$m$ - tuple of multiplication operators by the coordinate functions on the  Hilbert module $\mathcal M$ belongs to $B_1(\Omega)$. Then the results from \cite{misra-CD} apply
and show that the curvature is a complete unitary invariant. Therefore, if  the quotient module belongs to $B_1(\mathcal Z)$, we need to compute its curvature. It is shown in \cite{misra-rgdgm} that the curvature of $\mathcal Q$ is the restriction of the curvature $(1,1)$ form of $\mathcal M$ to the hyper-surface $\mathcal Z$ followed by a projection to the $(1,1)$ forms on $\mathcal Z$. 

The submodule in \cite{misra-DMV}  is taken to be the (maximal) set of functions which vanish to some given order { $k$} on the hypersurface { ${\mathcal Z}$}. As in the previous case, two descriptions are provided for the quotient module. The first one, produces a Hilbert space of holomorphic functions taking values in $\mathbb C^{k}$ via what is known as the jet construction. The kernel function now takes values in $k\times k$ matrices.  The second one provides a rank { $k$}  holomorphic Hermitian vector bundle. Although, describing a complete set of invaraints is much more complicated in this case { (cf. \cite{misra-qmtrans})}.

In the paper \cite{misra-CD3}, it was  observed that all the submodules of the Hardy module are  isomorphic to the Hardy module, that is, there exists an intertwining unitary module map between them.  Applying the von Neumann-Wold decomposition, it is not difficult to obtain description of all the submodules of the Hardy module. Following this new proof of the Beurling's theorem, it was natural determine isomorphism classes of submodules of other Hilbert modules. For instance, the situation is much more complicated for the Hardy module in $m$ - variables. Indeed, the submodule \index{Wold decomposition}
$$
H^2_0(\mathbb D^2) :=\{f \in H^2(\mathbb D^2)\mid f(0,0)=0\}
$$
is not equivalent to the Hardy module. This is easy to see: dimension of the joint kernel $\cap_{i=1,2} \ker M_i^*$ of the two multiplication operators on $H^2(\mathbb D^2)$ is $1$, while it is $2$ on $H^2_0(\mathbb D^2)$.

The fundamental question of which submodules of a Hilbert module are equivalent was answered in \cite{misra-D-P-S-Y} after making reasonable hypothesis on the nature of the submodule. A different approach to the same problem is outlined in \cite{misra-SBisM}. A sheaf model to study submodules like the Hardy module $H^2_0(\mathbb D^2)$ of functions vanishing at $(0,0)$ is given in \cite{misra-BMP}.

\subsection{Flag structure} \index{flat structure}
Fix a bounded planar domain $\Omega$. Let $E$ be a holomorphic Hermitian vector bundle of rank $n$ in $\Omega \times \mathcal H$. By a well-known theorem of Grauert, every holomorphic vector bundle over a plane domain is trivial. So, there exists a holomorphic frame  $\gamma_1, \ldots , \gamma_n: \Omega \to \mathcal H$. A Hermitian metric on $E$ relative to this frame is given by the formula 
$G_{\gamma}(w)= \left( \langle \gamma_i(w), \gamma_j(w)\rangle \right)$. The curvature of the vector bundle $E$ is a complex $(1,1)$ form which is given by the formula  \index{second fundamental form}
$$ \mathscr K_E(w)= \sum_{i,j=1}^m \mathscr K_{i,j}(w)\,  dw_i\wedge d\bar{w}_j,$$
where 
$$\mathscr K_{i,j}(w) = -\frac{\partial}{\partial \bar{w}_j}\big (G_{\gamma}^{-1}(w) (\frac{\partial}{\partial w_i} G_{\gamma}(w))\big ).$$  
It clearly depends on the choice of the frame ${\gamma}$ except when $n=1$. A complete set of invariants is given in 
\cite{misra-CD,misra-CD1}. However, these invariants are not easy to compute. So, finding a tractable set of invariants for a smaller class of vector bundles which is complete would be worthwhile. For instance, in the paper \cite{misra-JJKM}, irreducible holomorphic Hermitian vector bundles,  possessing a flag structure have been isolated. For these, the curvature together with the second fundamental form (relative to the flag) is a complete set of invariants. As an application, at least for $n=2$, it is shown that the homogeneous holomorphic Hermitian vector bundles are in this class. A complete description of these is then given. This is very similar to the case of $n=1$ except that now the second fundamental form associated with the flag has to be also considered along with the curvature. 
All the vector bundles in this class and the operators corresponding to them are irreducible. The flag structure they possess by definition  is rigid which  aids in the construction of a canonical model and in finding a complete set of unitary invariants. The study of commuting tuples of operators in the Cowen-Douglas class possessing a flag structure is under way. 

The definition of the smaller  class $\mathcal FB_2(\Omega)$ of operators in $B_2(\Omega)$ below is from \cite{misra-JJKM}. We will discuss the class $\mathcal FB_n(\Omega)$, $n >1$ separately at the end. 

\begin{definition}  We let $\mathcal FB_2(\Omega)$ denote the set of bounded linear operators $T$ for which we can find operators  $T_0,T_1$  in $\mathcal{B}_1(\Omega)$ and  a non-zero intertwiner
$S $ between $T_0$ and $T_1$, that is, $T_0S=S T_1$ so that 
$$T=\begin{pmatrix}
 T_0 & S \\
 0 & T_1 \\
\end{pmatrix}.$$
\end{definition}

An operator $T$ in  $B_2(\Omega)$ admits a decomposition of the form  $\begin{psmallmatrix}
T_0 & S \\
0 & T_1 \\
\end{psmallmatrix}$
for some pair of operators $T_0$ and $T_1$ in ${B}_1(\Omega)$  
(cf.  \cite[Theorem 1.49]{misra-jw}) . 
Conversely, an operator $T,$ which admits a decomposition of this form for some choice of $T_0, T_1$ in $B_1(\Omega)$ can be shown  to be in $B_2(\Omega).$
In defining the new class $\mathcal FB_2(\Omega)$, we are merely imposing one additional condition, namely that 
$T_0S=ST_1$.

An operator $T$ is in the class $\mathcal FB_2(\Omega)$ if and only if there exist a frame 
$\{\gamma_0,\gamma_1\}$ of the vector bundle $E_{T}$ such
that $\gamma_0(w)$ and $t_1(w):=\tfrac{\partial}{\partial w}\gamma_0(w)-\gamma_1(w)$ are orthogonal for all $w$ in $\Omega$.  This is also equivalent to the existence of a  frame 
$\{\gamma_0,\gamma_1\}$ of the vector bundle $E_{T}$ such that 
$$\tfrac{\partial}{\partial w}\|\gamma_0(w)\|^2=\langle\gamma_1(w),\gamma_0(w)\rangle,\quad w\in \Omega.$$ 
Our first main theorem on unitary classification is given below. 
  
\begin{theorem}\label{maint}
Let 
$$T=\begin{pmatrix}T_0 & S \\
0 & T_1 \\
\end{pmatrix},\qquad\tilde{T}=\begin{pmatrix}\tilde{T}_0 & \tilde S \\
0 & \tilde{T}_1 \\
\end{pmatrix}$$ 
be two operators in $\mathcal FB_2(\Omega)$. Also let $t_1$ and
$\tilde{t}_1$ be non-zero sections of the holomorphic Hermitian vector
bundles $E_{T_1}$ and $E_{\tilde{T}_1}$ respectively. The operators $T$ and
$\tilde{T}$ are equivalent if and only if
$\mathcal{K}_{T_0}=\mathcal{K}_{\tilde{T}_0}$ (or
$\mathcal{K}_{T_1}=\mathcal{K}_{\tilde{T}_1}$) and
$$\frac{\|S(t_1)\|^2}{\|t_1\|^2}= \frac{\|\tilde
S(\tilde{t}_1)\|^2}{\|\tilde{t}_1\|^2}.$$
\end{theorem}  
Cowen and Douglas point out in \cite{misra-CD1} that an operator in $B_1(\Omega)$ must be irreducible.  However, determining which operators in $B_n(\Omega)$ are irreducible is a formidable task.  It turns out that the operators in $\mathcal FB_2(\Omega)$ are always irreducible.  Indeed, if we assume $S$ is invertible, then $T$ is strongly irreducible. 
  
Recall that an operator $T$  in the Cowen-Douglas class $B_n(\Omega)$,  up to unitary equivalence, is the  adjoint of the multiplication operator {$M$} on a Hilbert space {$\mathcal H$} consisting of holomorphic functions on $\Omega^*:=\{\bar{w}: w \in \Omega\}$ possessing a reproducing kernel {$K$.}  What about operators in $\mathcal FB_n(\Omega)$? For $n=2,$ a model for these operators is described below.   

For an operator $T\in \mathcal FB_2(\Omega)$, there exists a  holomorphic frame $\gamma=(\gamma_0,\gamma_1)$ with the property $\gamma_1(w):=\tfrac{\partial}{\partial w}\gamma_0(w)-t_1(w)$ and that $t_1(w)$ is orthogonal to $\gamma_0(w)$, $w\in\Omega$, for some holomorphic map $t_1:\Omega \to \mathcal H$. In what follows, we fix a holomorphic frame with this property. Then the operator $T$ is unitarily equivalent to the adjoint of the multiplication operator $M$ on a Hilbert space
$\mathcal{H}_{\Gamma} \subseteq {\rm Hol}(\Omega^*, \mathbb C^2)$ possessing a reproducing kernel $K_{\Gamma}:\Omega^* \times \Omega^* \to \mathcal M_2(\mathbb  C)$. The details are in \cite{misra-JJKM}. It is easy to write down the kernel $K_\Gamma$ explictly: For $z,w\in\Omega^*$, we have 
\begin{eqnarray*}
K_{\Gamma}(z,w)&=&
\begin{pmatrix}
  \langle \gamma_0(\bar w),\gamma_0(\bar z)\rangle &
  \langle \gamma_1(\bar w),\gamma_0(\bar z)\rangle \\
   \langle \gamma_0(\bar w),\gamma_1(\bar z)\rangle &
   \langle \gamma_1(\bar w),\gamma_1(\bar z)\rangle \\
\end{pmatrix}\\
&=& \begin{pmatrix}
  \langle \gamma_0(\bar w),\gamma_0(\bar z)\rangle &
  \frac{\partial}{\partial \bar w}\langle \gamma_0(\bar w),\gamma_0(\bar z)\rangle \\
  \frac{\partial}{\partial z} \langle \gamma_0(\bar w),\gamma_0(\bar z)\rangle
  &
 \frac{\partial^2}{\partial z\partial \bar w} \langle \gamma_0(\bar w),
 \gamma_0(\bar z)\rangle+\langle t_1(\bar w),t_1(\bar z)\rangle \\
\end{pmatrix}, w\in\Omega.
\end{eqnarray*}
Setting $K_0(z,w)=\langle \gamma_0(\bar w),\gamma_0(\bar z)\rangle$
and $K_1(z,w)= \langle t_1(\bar w),t_1(\bar z)\rangle$, we see that the
reproducing kernel $K_{\Gamma}$ has the form:
\begin{equation} \label{main}
K_{\Gamma}(z,w)= 
\begin{pmatrix}
  {K_0}(z,w) & \frac{\partial}{\partial \bar w}{K_0}(z,w) \\
  \frac{\partial}{\partial z}{K_0}(z,w) & \frac{\partial^2}{\partial z\partial \bar w}{K_0}
  (z,w)+{K_1}(z,w) \\
\end{pmatrix}.
\end{equation}
  
All the irreducible homogeneous operators in $B_2(\mathbb D)$ belong to the class  $\mathcal F B_2(\mathbb D)$. An application of Theorem \ref{maint}  determines the curvature and the second fundamental form of these operators. It is then not hard to  identify these operators (up to unitary equivalence) as shown in \cite{misra-JJKM}.

\section{Some future directions and further thoughts}

In this second half of the paper, we discuss some of the topics mentioned only briefly in the first half. Along the way, we also mention some of the open problems in these topics. 

\subsection{Operators in the Cowen-Douglas class} 
In this section, we give a description of the local operators $\boldsymbol T_{| \mathcal N(w)}$, where 
$$\mathcal N(w) := \cap_{i,j=1}^m \ker(T_i - w_iI )(T_j - w_jI).$$ 
The matrix representation of these local operators contains one of the most important geometric invariant, namely, the curvature. Over the past three decades, this has been used to obtain various curvature inequalities. We also discuss a class of pure subnormal operators $T$ studied in detail by Abrahamse an Douglas, see \cite{misra-AD}. We describe briefly, a new realization of such operators of rank $1$ from the paper \cite{misra-Ramiz}, as multiplication operators on ordinary weighted Hardy spaces. For operators of rank $>1$, such a description perhaps  exists but has not been found yet.  

\subsubsection{Local operators} \label{locop}
Fix a $m$ - tuple of operators $\boldsymbol T$ in $\mathrm B_1(\Omega)$ and
let $N_{\boldsymbol T}(w)$ be the $m$-tuple of operators $(N_1(w),
\ldots, N_m(w)),$ where $N_i(w)= (T_{i}-w_{i}I)|_{\mathcal N(w)},$
$i=1,\ldots, m.$ Clearly, $N_i(w)N_j(w)=0$ for all $1\leq i,j \leq
m.$ 
Since $(T_i-w_iI)\gamma(w) =0$ and
$(T_i-w_iI) (\partial_j\gamma)(w) = \delta_{i,j} \gamma(w)$ for $1
\leq i,j \leq m$, we have  the matrix representation $N_k(w) = \left ( \begin{smallmatrix}
0 & e_k \\
 0 & \boldsymbol 0\\
\end{smallmatrix}\right ),$  where $e_k$ is the vector $(0,\ldots , 1, \ldots ,0)$ with the $1$ in the $k${\tt th} slot, $k=1,\ldots , m.$   
Representing $N_k(w)$ with respect to an orthonormal basis in
$\mathcal N(w)$, it is possible to read off the curvature of
$\boldsymbol T$ at $w$ using the relationship:
\begin{equation}\label{curvform}
\big (-\mathcal K_{\boldsymbol T}(w)^{\rm t}\big )^{-1}=\left( {\rm tr}\big ( N_{k}(w)\overline{N_{j}(w)}^{\rm t}
\big)\, \right)_{k,j=1}^m = A(w)^{\rm
t}\overline{A(w)},
\end{equation} 
where the $k${\tt th}-column of
$A(w)$ is the vector $\boldsymbol \alpha_k$ (depending on $w$)
which appears in the matrix representation of $N_k(w)$ with
respect to an appropriate choice of an orthonormal basis in $\mathcal N(w)$
which we describe below.

This formula is established for a pair of operators in $\mathrm
B_1(\Omega)$ (cf. \cite[Theorem 7]{misra-CD1}).  However, let us 
verify it for an $m$-tuple $\boldsymbol T$ in $\mathrm
B_1(\Omega)$ for any $m \geq1$ following \cite{misra-MisPal}. 

Fix $w_0$ in $\Omega$. We may assume without loss of generality
that $\|\gamma(w_0)\|=1$. The function $\langle \gamma(w),
\gamma(w_0)\rangle$ is invertible in some neighborhood of $w_0$.
Then setting $\hat{\gamma}(w):=  \langle \gamma(w),
\gamma(w_0)\rangle^{-1} \gamma(w)$, we see that
$$\langle \partial_k \hat{\gamma}(w_0), \gamma(w_0) \rangle = 0, \,\, k=1,2,\ldots, m.$$
Thus $\hat{\gamma}$ is another holomorphic section of $E$.  The
norms of the two sections $\gamma$ and $\hat{\gamma}$ differ by
the absolute square of a holomorphic function, that is
$$\tfrac{\|\hat{\gamma}(w)\|}{\|\gamma(w)\|} = |\langle \gamma(w),
\gamma(w_0)\rangle|.$$ 
Hence the curvature is independent of the
choice of the holomorphic section.
 Therefore, without loss of generality, we will prove the claim assuming,
for a fixed but arbitrary $w_0$ in $\Omega$, that
\begin{enumerate}
\item[{(i)}] $\|\gamma(w_0)\|=1$, \item[{(ii)}]
$\gamma(w_0)$ is orthogonal to $(\partial_k\gamma)(w_0)$,
$k=1,2,\ldots , m$.
\end{enumerate}

Let $G$ be the Grammian corresponding to the $m+1$ dimensional
space spanned by the vectors \index{Grammian}
$$\{\gamma(w_0),
(\partial_1\gamma)(w_0),  \ldots , (\partial_m \gamma)(w_0)\}.$$
This is just the space $\mathcal N(w_0)$. Let $v, w$ be any two
vectors in $\mathcal N(w_0)$. Find $\boldsymbol c=(c_0,\ldots,
c_m), \boldsymbol d=(d_0, \ldots, d_m)$ in $\mathbb C^{m+1}$ such
that $v=\sum_{i=0}^{m}c_i {\partial}_i\gamma(w_0)$ and
$w=\sum_{j=0}^{m}d_j{\partial}_j \gamma(w_0).$  Here
$(\partial_0\gamma)(w_0) = \gamma(w_0)$. We have
\begin{align*}\langle v, w\rangle&=\langle\sum_{i=0}^{m}c_i{\partial}_i\gamma(w_0),
\sum_{j=0}^{m}d_j {\partial}_j \gamma(w_0)\rangle\\
&= \langle G^{\rm t}(w_0)\boldsymbol c, \boldsymbol d\rangle_{\mathbb C^{m+1}}\\
&= \langle (G^{\rm t})^{\frac{1}{2}}(w_0)\boldsymbol c, (G^{\rm
t})^{\frac{1}{2}}(w_0)\boldsymbol d\rangle_{\mathbb C^{m+1}}.
\end{align*}
Let $\{\varepsilon_i\}_{i=0}^{m}$ be the standard orthonormal basis for
$\mathbb C^{m+1}$. Also, let $(G^{\rm
t})^{-\frac{1}{2}}(w_0)\varepsilon_i:=\boldsymbol \alpha_i(w_0)$, where
$\boldsymbol \alpha_i(j)(w_0) = \alpha_{j i}(w_0)$, $i=0,1,\ldots,
m$. We see that the vectors  $\varepsilon_i:=\sum_{j=0}^m
\alpha_{ji} (\partial_j \gamma)(w_0)$, $i=0,1, \ldots ,m,$ form an
orthonormal basis in $\mathcal N(w_0)$:  \begin{align*}\langle
\varepsilon_i, \varepsilon_l\rangle &=\big\langle
\sum_{j=0}^{m}\alpha_{j i}{\partial}_j \gamma(w_0),
\sum_{p=0}^{m}\alpha_{p l}{\partial}_p \gamma(w_0)\big\rangle\\&=
\langle(G^{\rm t})^{-\frac{1}{2}}(w_0)\boldsymbol \alpha_i, (G^{\rm
t})^{-\frac{1}{2}}(w_0)\boldsymbol
\alpha_l\rangle\\&=\delta_{il},\end{align*} where
$\delta_{il}$ is the Kronecker delta. Since $N_k\big
(\,(\partial_j\gamma)(w_0) \,\big )= \gamma(w_0)$ for $j=k$ and
$0$ otherwise, we have $N_k(\varepsilon_i) = \Big (
\begin{smallmatrix} 0 & \boldsymbol \alpha_k^{\rm t}\\ 0 & 0
\end{smallmatrix} \Big )$. Hence
\begin{align*}
{\rm tr}\big ( N_i(w_0) N_j^*(w_0) \big ) &=  \boldsymbol{\alpha_i}(w_0)^{\rm t} \overline{\boldsymbol \alpha}_j(w_0)\\
&= \big ( (G^{\rm t})^{-\frac{1}{2}}(w_0)\varepsilon_i\big )^{\rm t}\overline{\big ( (G^{\rm t})^{-\frac{1}{2}}(w_0)\varepsilon_j\big )}\\
&= \langle {G}^{-\frac{1}{2}}(w_0)\varepsilon_i , {G}^{-\frac{1}{2}}(w_0)\varepsilon_j
\rangle = {(G^{\rm t})}^{-1}(w_0)_{i,j}.
\end{align*}
Since the curvature, computed with respect to the holomorphic
section $\gamma$ satisfying the conditions {(i)} and
{(ii)}, is of the form
\begin{align*}
-\mathcal K_{\boldsymbol T}(w_0)_{i,j} &= \frac{\partial^2}{\partial w_i {\partial} \bar{w}_j}\log \|\gamma(w)\|^2_{|w=w_0}\\
&= \Big ( \frac{\|\gamma(w)\|^2 \big ( \frac{\partial^2
\gamma}{\partial w_i \partial\bar{w}_j} \big )(w) - \big
(\frac{\partial \gamma}{\partial w_i}\big )(w)  \big (
\frac{\partial \gamma }{\partial \bar{w}_j}\big ) (w)}
{\|\gamma(w)\|^4}\Big )_{|w=w_0}\\
&= \big ( \frac{\partial^2 \gamma}{\partial w_i \partial
\bar{w}_j} \big )(w_0)=G(w_0)_{i,j}, \end{align*} we have verified
the claim \eqref{curvform}.

The local description of the $m$ - tuple of operators $\boldsymbol T$ shows that the curvature is indeed obtained from the holomorphic frame and the first order derivatives using the Gram-Schmidt orthonormalization. The following theorem was proved for $m=2$ in ({cf. \cite[Theorem 7]{misra-CD1}}). However, for any natural number $m$, the proof is
evident from the preceding discussion. The case of a $m$ -tuple of operators in $B_n(\Omega)$ for an arbitrary $n \in \mathbb N$ is discussed in \cite{misra-MR}.
\index{Gram-Schmidt orthonormalization}
\begin{theorem}
Two $m$-tuples of operators $\boldsymbol T$ and
$\tilde{\boldsymbol T}$ in $\mathrm B_1(\Omega)$ are unitarily
equivalent if and only if $N_k(w)$ and $\tilde{N}_k(w),$ $1\leq k \leq m$, are
simultaneously unitarily equivalent for $w$ in some open subset of
$\Omega$. 
\end{theorem}
\begin{proof}
Let us fix an arbitrary point $w$ in $\Omega$. In what follows,
the dependence on this $w$ is implicit. Suppose that there exists
a unitary operator $U:\mathcal N\rightarrow \widetilde{\mathcal
N}$ such that $UN_i=\tilde{N_i}U$, $i= 1,\ldots,m.$
For $1 \leq i, j \leq m,$ we have
\begin{align*}{\rm tr}\big ( \tilde{N_i}
\tilde{N_j}^* \big )&={\rm tr}\big (\big(U N_iU^*\big)\big(U
N_jU^*\big)^* \big )\\&={\rm tr}\big ( UN_i N_j^* U^*\big
)\\&={\rm tr}\big ( N_i N_j^* U^*U\big )\\&={\rm tr}\big ( N_i
N_j^* \big ).
\end{align*}
Thus the curvature of the operators $\boldsymbol T$ and $
{\tilde{\boldsymbol T}}$ coincide making  them unitarily
equivalent proving the Theorem in one direction. In the other
direction, observe that if the operators $\boldsymbol T$ and $
{\tilde{\boldsymbol T}}$ are unitarily equivalent then this unitary  
must map $\mathcal N$ to $\tilde{\mathcal N}$. Thus the restriction of $U$ to the subspace
$\mathcal N$ intertwines $N_k$ and $\tilde{N}_k$ simultaneously
for $k=1,\cdots ,m$.
\end{proof}

\subsubsection{Pure subnormal operators} \index{pure subnormal operator}
\index{minimal normal extension}\index{bilateral shift}\index{pure isometry}
The unilateral shift $U_+$ or the multiplication $M$ by the coordinate function on the Hardy space $H^2(\mathbb D)$ is a very special kind of subnormal operator
in that it is a pure isometry. The spectrum of its minimal normal extension, namely, the bi-lateral shift $U$ or the operator of multiplication $M$ by the coordinate function on $L^2(\mathbb T)$  is the unit circle $\mathbb T = \partial \overbar{\mathbb D} = \partial \sigma(U_+)$. These two properties determine such an operator uniquely, that is, if a pair consisting of a  pure isometry $S$ and its minimal normal extension $N$ has the spectral inclusion property: $\sigma(N) \subseteq \partial \sigma(S)$, then $S$ must be unitarily equivalent to the direct sum of a number of copies of the operator $U_+$. The situation for the annulus is more complicated and was investigated in \cite{misra-DSannulus}. Following this, Abrahamse and Douglas initiated the study of a class of pure subnormal operators $S$ which share this property. Thus the spectrum $\sigma(S)$ of the operator  $S$ is assumed to be a subset of the closure $\overbar{\Omega}$ of a bounded domain $\Omega$  and the spectrum 
$\sigma(N)$ of its normal extension $N$ is contained in $\partial \sigma(S)$. 

Let $\Omega$ be a bounded domain in $\mathbb C$ and let $\mathcal O(\Omega)$ 
be the space of holomorphic functions on $\Omega$.  For $f\in \mathcal O(\Omega)$, 
the real analytic function $|f|$ is subharmonic and hence admits a least harmonic majorant, which is the function   \index{least harmonic majorant}
 $$
 u_f(z):=\inf  \{u(z) : |f| \leq u, u \mbox{\rm~ is harmonic on}\, \,\Omega\},  
 $$
that is, $u_f$ is either harmonic or infinity throughout $\Omega$.

Fix a point $w\in \Omega$. The Hardy space on $\Omega$ is  defined to be the Hilbert space 
$$
H^2(\Omega) :=\{ f\in \mathcal O(\Omega): u_f(w) < \infty\}.
$$
It is easily verified that $\|f\|:= u_f(w)$ defines a norm and a different choice of $w\in \Omega$ induces an equivalent topology. 
Let ${\mathcal O}_{\mathfrak h}(\Omega)$ denote the space of holomorphic functions $f$ defined on $\Omega$ taking values in some Hilbert space $\mathfrak h$. Again, $z\mapsto \|f(z)\|_{\mathfrak h}$ is a subharmoinic function and admits a least harmonic majorant 
$u_f$.  We define the Hardy space 
$$H^2_{\mathfrak h}(\Omega):=\{f\in {\mathcal O}_{\mathfrak h}(\Omega): \|f\|_w := u_f(w) < \infty\}$$
as before.  \index{harmonic measure}
Let $m$ be the harmonic measure relative to the point $w\in \Omega$ and let $L^2(\partial \Omega, m)$ denote the space of square integrable functions defined on 
$\partial \Omega$ with respect to $m$.  The closed subspace 
$$H^2(\partial \Omega):= \{ f \in L^2(\partial \Omega, m): \int_{\partial \Omega}  f g\, dm = 0, g\in \mathcal A(\Omega) \},$$
where $\mathcal A(\Omega)$ is the space of all functions which are holomorphic in some open neighbourhood of the closed set $\overline{\Omega}$, is the Hardy space of $\partial \Omega$.   
Let $L^2_{\mathfrak h}(\partial \Omega)$ denote the space of square integrable functions, with respect to the measure $m$, taking values in the Hilbert space $\mathfrak h$.  
The Hardy space $H^2_{\mathfrak h}(\partial \Omega)$ is the closed subspace of  
$L^2_{\mathfrak h}(\partial \Omega)$ consisting of those functions in 
$L^2_{\mathfrak h}(\partial \Omega)$ for which $\int_{\partial \Omega}  f g\, dm = 0$, $g\in \mathcal A(\Omega)$.  

\emph{Boundary values.} 
A function $f$ in the Hardy space admits a boundary value 
$\hat{f}$.  This means that the $\lim_{z\to \lambda} f(z)$ exists (almost everywhere relative to $m$)  as $z$  approaches  $\lambda\in \partial \Omega$ through any non-tangential path in $\Omega$.   
Define the function $\hat{f}:\partial \Omega\to \mathbb C$ by setting $\hat{f}(\lambda) = \lim_{z\to \lambda} f(z)$.  It then follows that the map $f \mapsto \hat{f}$ is an isometric isomorphism between the two Hardy spaces $H^2(\Omega)$ and $H^2(\partial \Omega)$, see \cite{misra-RudinHp}. This correspondence works for the case of  
$H_{\mathfrak h}^2(\Omega)$ and $H_{\mathfrak h}^2(\partial \Omega)$ as well.

A topological space $E$ is said to be a vector bundle of rank $n$ over $\Omega$ if there is a continuous map $p:E \to \Omega$  such that $E_z:=p^{-1}(z)$ is a linear space. A co-ordinate chart relative to  an open cover $\{U_s\}$ of the domain $\Omega$ is a set of homeomorphisms $\varphi_s:U_s\times\mathfrak h \to p^{-1}(U_s)$ such that 
$\varphi_s(z,k)$ is in $E_z$ for any fixed $z\in U_s$, the restriction  $\varphi_s(z):={\varphi_s|}_{\{z\}\times \mathfrak h}$  is a continuous linear isomorphism.  Therefore, $\varphi_{st}(z) : = \varphi^{-1}_s(z) \varphi_t(z)$ is in ${\rm GL}(\mathfrak h)$, $z\in U_s\cap U_t$. If the transition functions $\varphi_{st}$  are holomorphic for some choice of co-ordinate functions $\varphi_s$, then the vector bundle $E$ is said to be a holomorphic vector bundle. A section $f$ of $E$ is said to be holomorphic if $\varphi_s^{-1}f $ is a holomorphic function from $U_s$ to $\mathfrak h$. Finally, if $\varphi_s^{-1}(z) \varphi_t(z)$ are chosen to be unitary, $z\in U_s\cap U_t$, then one says that $E$ is a flat unitary vector bundle.  Thus for a flat unitary vector bundle $E$,  we have that 
$$\|\varphi_s(z)^{-1}f(z)\|_{\mathfrak h} =  \|\varphi_t(z)^{-1}f(z)\|_{\mathfrak h},
\quad z\in U_s\cap U_t.$$ 
This means that the functions 
$$h^{(s)}_f(z):= \|\varphi_s(z)^{-1}f(z)\|_{\mathfrak h},\quad z\in U_s,$$ 
agree on the overlaps and therefore define a function $h^E_f$ on all of $\Omega$.  
Since $z\mapsto \varphi_s(z)^{-1}f(z)$ is holomorphic and subharmonicity is a local property, it follows that $h_f$ is subharmonic on $\Omega$.  As before, the function $h^E_f$ admits a least harmonic majorant $u^E_f$ and the Hardy space $H_E^2(\Omega)$ is the subspace of holomorphic sections $f$ of $E$ for which 
$$\|f\|^E_w:=u^E_f(w) < \infty$$ 
for some fixed but arbitrary point $w\in \Omega$. 

A function $f$ which is bounded and holomorphic on $\Omega$ defines a bounded linear operator $M_f$ on the Hardy space $H_{\mathfrak h}^2(\Omega)$ via pointwise multiplication.    Let $T_{\mathfrak h}$ denote the multiplication by the coordinate function $z$.  It is then evident that $T_{\mathfrak h}$ is unitarily equivalent to the tensor product $T\otimes I_{\mathfrak h}$, where $T$ is the multiplication induced by the coordinate function $z$ on the Hardy space $H^2(\Omega)$.  
The Hardy space $H^2_E(\Omega)$ is also invariant under multiplication by any function which is holomorphic and bounded on $\Omega$. In particular, let $T_E$ be the operator of multiplication by $z$ on the Hardy space $H^2_E(\Omega)$. All the three theorems listed below are from \cite{misra-AD}. 
\begin{theorem}
If $\mathfrak h$ is a Hilbert space and $E$ is a flat unitary vector bundle over $\Omega$ of rank $\dim \mathfrak h$, then $T_E$ is similar to $T_{\mathfrak h}$, that is there is a bounded invertible operator $L:H_{\mathfrak h}^2(\Omega) \to H^2_E(\Omega)$ with the property $L^{-1}T_EL=T_{\mathfrak h}$.   
\end{theorem}
\begin{theorem}
The operators $T_E$ and $T_F$ are unitarily equivalent if and only if $E$ and $F$ are equivalent as flat unitary vector bundles.
\end{theorem}
\begin{theorem} 
If $E$ is a flat unitary vector bundle of dimension $n$, then for $z\in \Omega$, the dimension of the kernel of $(T_E-z)^*$ is $n$.
\end{theorem}

The operator $T_E^*$ on $H_E^2(\Omega)$, where $E$ is a flat unitary line bundle is in $B_1(\Omega^*)$, see \cite[Corollary 2.1]{misra-GM}. Consequently, it defines  a holomorphic Hermitian line bundle, say $\mathcal E$. One of the main questions that remains unanswered is the relationship between the flat unitary bundle $E$ and the holomorphic Hermitian bundle $\mathcal E$. Thus we are asking which operators in $B_1(\Omega)$ are   pure subnormal operator with the spectral inclusion property and conversely, how to find a model for such an operator in $B_1(\Omega)$.  Any answer involving the intrinsic complex geometry will be very interesting. In spite of the substantial work in \cite{misra-DY} and \cite{misra-LiChen}, this question remains elusive. We discuss this a little more in Section \ref{CurvInR}.  Also, the question of commuting tuples of subnormal operators has not been addressed yet. These are very interesting directions for future research. 

\index{covering map}\index{automorphism}
There is yet another way, which exploits the covering map $\pi:\mathbb D \to \Omega$, to define the Hardy space on $\Omega$.  
Let $G$ be the group of deck transformations for the covering map $\pi$.  Thus $G$ consists of those bi-holomorphic automorphisms $\varphi:\mathbb D \to \mathbb D$ satisfying $\pi\circ \varphi = \pi$.  A function $f$ on $\mathbb D$ is said to be $G$-automorphic if $f\circ \varphi = f$ for $\varphi \in G$.  A function $f$ defined on $\mathbb D$ is $G$ -automorphic if and only if it is of the form $g\circ \pi$ for some function $g$ defined on $\Omega$.   

Let $w=\pi(0)$ and $m$ be the harmonic measure on $\partial \Omega$ relative to $w$.
This means that $m$ is supported on $\partial \Omega$ and if $u$ is a function which is continuous on $\overline{\Omega}$ and harmonic on $\Omega$, then  
$$u(w) = \int_{\partial \Omega} u\,dm.$$  
For any $f\in L^1(\partial \Omega, m)$, it may be shown that 
$$\int_{\partial \Omega} f\,dm = \int_{\partial \mathbb D} f\circ \pi d\mu,$$ 
where $\mu$ is the Lebesgue measure on $\partial \mathbb D$.  Thus $m=\mu\circ \pi^{-1}$.

The Hilbert spaces $L^2_{\mathfrak h}(\partial \Omega)$,  $H^2_{\mathfrak h}(\Omega)$ and $H^2_{\mathfrak h}(\partial \Omega)$ now lift to closed subspaces of the respective Hilbert spaces on the unit disc $\mathbb D$ via the map $f \mapsto f\circ \pi$.

The dual $\hat{G}$ of the  group $G$ is the group of homomorphisms from $G$ into the circle group $\mathbb T$.  For each character $\alpha \in \hat{G}$, define the Hardy space 
$$
H^2_\alpha(\mathbb D):= \{f \in H^2(\mathbb D): f\circ \varphi = \alpha(\varphi) f \}.
$$
This is a closed subspace of $H^2(\mathbb D)$ which is invariant under any function which is  $G$ - automorphic,  holomorphic and bounded on  $\mathbb D$. The quotient map $\pi$ is such a function and we set $T_\alpha$ to be the operator of multiplication by $\pi$ on $H^2_\alpha(\mathbb D)$.   

We can define vector valued Hardy spaces $H^2_\alpha(\mathbb D, \mathfrak h)$ in a similar manner except that $\alpha$ is in $\mbox{\rm Hom}(G, \mathcal U(\mathfrak h))$, the group of unitary representations of $G$ on the Hilbert space $\mathfrak h$.  The operator $T_\alpha$ on $H^2_\alpha(\mathbb D, \mathfrak h)$ is multiplication by  $\pi$ as before.  
Since $G$, in this case, is isomorphic to the fundamental group of $\Omega$, there is a bijective correspondence between equivalence classes of the unitary representations of $G$ on $\mathfrak h$  and the equivalence classes of flat unitary vector bundles on $\Omega$. 
Again, the two theorems stated below are from \cite{misra-AD}.

\begin{theorem}
The operators $T_\alpha$ and $T_\beta$ are unitarily equivalent if and only if $\alpha$ and $\beta$ are equivalent representations.  
\end{theorem}

\begin{theorem}
If $E$ is the flat unitary vector bundle on $\Omega$ determined by the representation $\alpha$, then the operator $T_E$ is unitarily equivalent to  the operator $T_\alpha$.      
\end{theorem}

The choice of the harmonic measure in the definition of the Hardy space over the domain $\Omega$  may appear to be somewhat arbitrary but it has the advantage of being a conformal invariant and it is closely related to the Greens' function $g$ of $\Omega$ with pole at $t$:  \index{Green's function}\index{directional derivative}
\begin{align*}
dm(z) &= -\frac{1}{2\pi} \frac{\partial}{\partial \eta _z}\big(g(z,t)\big)ds(z),\;\; z \in \partial \Omega, 
\end{align*} where $g(z,t)$ denotes the Green's function for the domain $\Omega$ at the point $t$ and $\frac{\partial}{\partial \eta _z} $ is the directional derivative along the outward normal direction (with respect to the positively oriented $\partial \Omega$). Exploiting this, in the paper \cite{misra-MR}, a description of the bundle shift is given as multiplication operators on Hardy spaces defined on the domain $\Omega$ with respect to a weighted arc length measure. Briefly, starting with  a positive continuous function $\lambda$ on $\partial \Omega$, define the weighted measure $\lambda ds$ on $\partial \Omega$.  Since the harmonic measure $m$ is boundedly mutually absolutely continuous with respect to the arc length measure $ds$, it follows that  $\big( H^2(\partial\Omega), \lambda ds\big)$ acquires the structure of a Hilbert space and the operator $M$ on it is a pure, rationally cyclic, subnormal operator with spectrum equal to $\overbar{\Omega}$ and the spectrum of the minimal normal extension is $\partial\Omega.$ Consequently, the operator $M$ on $\big( H^2(\partial\Omega), \lambda ds\big)$ must be unitarily equivalent to the bundle shift $T_{\alpha}$ on $\big( H^2_{\alpha}(\Omega), dm\big)$ for some character $\alpha.$ In the paper \cite{misra-MR}, this character has been explicitly described in Equation (2.2). It is then shown that given any character $\alpha$, there exists a function $\lambda$ such that the operator $M$ on $H^2(\partial \Omega, \lambda ds)$ is unitarily equivalent to the bundle shift $T_\alpha$.

\subsection{Curvature inequalities} 
The local description of the Cowen-Douglas operators $\boldsymbol T$, naturally lead to curvature inequalities relate to many well known extremal problems assuming that the homomorphism induced by the operator $\boldsymbol T$ is contractive. Among other things, we discuss a) the question of uniqueness of the extremal operators and b) if the curvature inequalities with additional hypothesis implies contractivity.

\subsubsection{The Douglas question}\label{CurvInR}
Let 
$$\mathcal B[w]:=\{f\in \mathcal O(\Omega): \|f\|_\infty < 1, f(w)=0\}.$$  
It is well-known that the extremal problem 
\begin{equation}\label{extrOmega}
\sup\{|f^\prime(w)|: f\in \mathcal B[w]\}
\end{equation}\index{Ahlfor's function}\index{branch point}\index{S\"{z}ego kernel}
admits a solution, say, $F_w \in \mathcal B[w]$. The function $F_w$ is called the Ahlfor's function and maps $\Omega$ onto $\mathbb D$  in a $n$ to $1$ fashion if the connectivity of the region $\Omega$ is $n$.  Indeed, $F_w$ is a branched covering map with branch point at $w$. Also, $F^\prime_w(w)$ is a real analytic function and polarizing it, we get a new function, holomorphic in the first variable and anti-holomorphic in the second, which is the S\"{z}ego kernel of the domain $\Omega$. It is also the reproducing kernel of the Hardy space $H^2(\Omega, ds)$, which we denote by the symbol $S_\Omega$. The relationship of the Hardy space with the solution to the extremal problem \eqref{extrOmega} yields a curvature inequality for contractive homomorphisms $\varrho_T$ induced by operators $T$ in $B_1(\Omega)$, where $\varrho_T(r) = r(T)$, $r \in  {\rm Rat}(\Omega)$.  Recall that the local operator $N(w) + w I = T_{|\mathcal N(w)}$ is of the form $\left (\begin{smallmatrix} w & h(w)\\ 0 & w\end{smallmatrix} \right )$, where
$$h(w) = \big (-\mathscr K_T(w)\big )^{-\tfrac{1}{2}}.$$ 
Consequently, 
$$r(T_{| \mathcal N(w)})  = \left (\begin{smallmatrix} r(w) & r^\prime(w) h(w)\\ 0 & r(w)\end{smallmatrix} \right ).$$ 
We have the obvious inequalities
\begin{align*}
h(w) &\sup\{ |r^\prime(w)| ; r\in {\rm Rat}(\Omega),\, r(w) = 0\}\\ 
&=\sup\{\|\left (\begin{smallmatrix} 0 & r^\prime(w) h(w)\\ 0 & 0\end{smallmatrix} \right )\|: r\in {\rm Rat}(\Omega), r(w) = 0\} \\ 
&\leq  \sup\{ \|\left (\begin{smallmatrix} r(w) & r^\prime(w) h(w)\\ 0 & r(w)\end{smallmatrix} \right )\| : r \in {\rm Rat}(\Omega)\}\\
&\leq \sup\{\|\varrho_T(r) \| :  r \in {\rm Rat}(\Omega)\}
\end{align*}
The contractivity of $\varrho_T$  then gives the curvature inequality
\begin{eqnarray} \label{cuinO}
\mathscr K_T(w) &=& - (\tfrac{1}{h(w)})^2  \nonumber \\
&\leq& - \big (\sup\{ |r^\prime(w)| ; r\in {\rm Rat}(\Omega),\, r(w) = 0\}\big )^2 \nonumber\\
&\leq & - \sup\{|f^\prime(w)|: f\in \mathcal B[w]\} \nonumber
\\ &\leq& - S_\Omega(w,w)^2.
\end{eqnarray}
If $\Omega$ is the unit disc, then from the Schwarz Lemma, it follows that $F^\prime_w(w) = \tfrac{1}{(1-|w|^2)}$. For the unit disc, taking $T=M^*$, the adjoint of the multiplication operator on the Hardy space $H^2(\mathbb D)$, we see that we have equality for all $w$ in $\mathbb D$ in the string of inequalities \eqref{cuinO}.  One can easily see that it makes no difference if we replace the unit disc by any simply connected domain. However, if $\Omega$ is not simply connected, it is shown in \cite{misra-SUITA} that the inequality is strict if we take the operator $T$ to be the adjoint of the multiplication operator on the Hardy space $H^2(\Omega, ds)$.  

Let $K$ be a positive definite kernel on $\Omega$. Assume that the adjoint of the multiplication operator $M$ is in $B_1(\Omega^*)$. What we have shown is that if $\varrho_M$ is contractive, then we have the inequality 
$$\mathcal K(\bar{w}): = - \frac{\partial^2}{\partial w\,\partial \bar{w}} \log K(w,w) \leq  - S_\Omega(w,w)^2,\,\, w\in \Omega,$$
where $\mathcal K$ denotes the curvature of the operator $M^*$ on $(\mathcal H,K)$. Since the obvious candidate which might have served as an extremal operator, namely, the multiplication operator on the Hardy space $H^2(\Omega, ds)$ is not extremal, one may ask 
 if for a fixed but arbitrary  $w_0 \in \Omega$, there exists a kernel $K_0$, depending on $w_0$,   for which $\mathcal K(w_0) = - F_{w_0}(w_0)^2$.  The existence of a kernel $K_0$ with this property was established in \cite{misra-GM}, see also \cite{misra-Ramiz}.  In the paper \cite{misra-Ramiz}, more is proved. For instance, it is established that the multiplication operator on $(\mathcal H, K_0)$ is uniquely determined answering the question of Douglas within the class of pure subnormal operators $T$ with  $\sigma(T) \subseteq \overbar{\Omega}$ and $\sigma(N) \subseteq \partial \sigma(T)$, where $N$ is the minimal normal extension of $T$. This is Theorem 2.6 of \cite{misra-Ramiz}.
The question of Douglas has an affirmative answer in a much larger class of operators when considering contractions in $B_1(\mathbb D)$. The main theorem in \cite{misra-MR} is reproduced below. 
\begin{theorem}\label{uniqueness in unit disc}
Fix an arbitrary point $\zeta \in \mathbb D.$ Let $T$ be an operator in $B_1(\mathbb D)$ such that $T^*$ is a $2$ hyper-contraction. Suppose that the operator $(\phi_{\zeta}(T))^*$ has the wandering subspace property for an automorphism $\phi_{\zeta}$ of the unit disc $\mathbb D$ mapping $\zeta$ to $0.$ If $\mathcal K_T(\zeta)=-(1- |\zeta|^2)^{-2},$ then $T$ must be unitarily equivalent to $U_+^*,$ the  backward shift operator.
\end{theorem}
We believe, the condition in the Theorem requiring the operator $(\phi_{\zeta}(T))^*$ to have the wandering subspace property must follow from the assumption that $T$ is in $B_1(\mathbb D)$. We haven't been able to prove this yet. 

Continuing our previous question of which holomorphic vector bundles come from flat unitary ones, one might have imagined that the operators in $B_1(\Omega)$ which are extremal would do the trick. However, it is shown in \cite{misra-Ramiz} that many of the  bundle shifts cannot be extremal at any point in the domain. This is discussed at the end of Section 3 of the paper \cite{misra-Ramiz}. So, this question remains open.  

\subsubsection{Infinite divisibility} \index{infinite divisibility}
For a contraction in the Cowen-Douglas class $B_1(\mathbb D)$, we have established that  $\mathscr K_T(w) \leq \mathscr K_{S^*}(w),\, w\in \mathbb D$, where $S$ is the forward shift operator.  Clearly, this is equivalent to saying that the restriction of $T$ to the $2$ - dimensional subspace $\mathcal N_T(w)$ spanned by the two vectors $\gamma_T(w), \gamma_T^{\,\prime}(w)$ is contractive. Therefore, it is unreasonable to expect that the curvature inequality for an operator $T$ would force it to be contractive. Examples are given in \cite{misra-BKM}. A natural question is to ask if the curvature inequality can be strengthened to obtain contractivity. Let $K$ be a positive definite kernel and the adjoint of the multiplication operator $M$ is in $B_1(\mathbb D)$. Clearly, 
$${K}^\ddag(z,w):=(1- z \bar{w})K(z,\bar{w})$$ 
need not be nnd (unless $M$ is contractive), however, it is Hermitian symmetric, i.e. 
$${K^\ddag(z,w)} = \overline{K^\ddag(w,z)}.$$ 
Now, observe that the curvature inequality is equivalent to the inequality 
$$\frac{\partial^2}{\partial w \partial\bar{w}}\log {{K^\ddag}(w,w)} \geq 0.$$ 
But the function 
$$\frac{\partial^2}{\partial w \partial\bar{w}}\log {{K^\ddag}(w,w)}$$ 
is real analytic, its polarization is a function of two complex variables and it is Hermitian symmetric in those variables. \index{Hermitian symmetric function}
Thus we ask what if we make the stronger assumption that 
$$\frac{\partial^2}{\partial w \partial\bar{w}}\log {{K^\ddag}(w,w)}$$ 
is nnd. For a Hermitian symmetric function $K$, we write $K \succeq 0$ to mean that 
$K$ is nnd. Similarly, this stronger form of the curvature inequality implies ${K^\ddag}$ must be infinitely divisible, that is, not only ${K^\ddag}$ is nnd but all its positive real powers ${K^\ddag}^t$ are nnd as well. In particular, it follows that the operator $M$ must be contractive. This is Corollary 4.2 of the paper 
\cite{misra-BKM} which is reproduced below. For two Hermitian symmetric functions $K_1$ and $K_2$, the inequality $K_1 \preceq K_2$ means $K_1 - K_2$ is nnd. 

\begin{theorem}\label{mainc1}
Let $K$ be a positive definite kernel on the open unit disc $\mathbb D$.
Assume that the adjoint $M^*$ of the multiplication operator
$M$ on the reproducing kernel Hilbert space $\mathcal H_K$
belongs to $B_1(\mathbb D)$. The  function
$$\frac{\partial^2}{\partial w\,\partial\bar{w}} \log \big ((1 -|w|^2)K(w,w) \big )$$ 
is positive definite, or equivalently
$$ \mathcal K_{M^*}(w) \preceq \mathcal K_{S^*}(w),\,\,w\in\mathbb D,$$
if and only if the multiplication operator $M$ is an infinitely divisible contraction.
\end{theorem}

\subsubsection{The multi-variable case} \index{row contraction}
We say that a commuting tuple of multiplication operators $\boldsymbol M$ is
an infinitely divisible row contraction if $(1-\langle z,w\rangle) K(z,w)$
is infinitely divisible, that is, 
$$\big ((1-\langle z,w\rangle) K(z,w) \big )^t$$
is positive definite for all $t>0$.

Recall that ${\boldsymbol R}_m^*$ is the adjoint of the joint weighted shift operator on the
Drury-Arveson space $H^2_m$. The
following theorem is a characterization of infinitely divisible
row contractions.\index{Drury-Arveson space}

\begin{proposition} (\cite[Corollary 4.3]{misra-BKM})
Let $K$ be a positive definite kernel on the Euclidean ball
$\mathbb B_m$. Assume that the adjoint $\boldsymbol M^*$ of the
multiplication operator $\boldsymbol M$ on the reproducing kernel Hilbert
space $\mathcal H_K$ belongs to $B_1(\mathbb B_m)$. The  function
$$\big (\!\big (\frac{\partial^2}{\partial w_i\,\partial\bar{w_j}}
\log \big ((1 -\langle  w,  w \rangle)K( w, w) \big )\,\big
)\!\big )_{i,j=1}^m,\quad w\in \mathbb B_m,$$ 
is positive definite, or equivalently
$$ {\mathscr K}_{\boldsymbol M^*}( w) \preceq {\mathscr K}_{\boldsymbol R_m^*}( w),
\quad w\in\mathbb B_m,$$
if and only if the multiplication operator $\boldsymbol M$ is an infinitely
divisible row contraction.
\end{proposition}

In the case of the polydisc, we say a commuting tuple $\boldsymbol M$ of multiplication
by the co-ordinate functions acting on
the Hilbert space $\mathcal H_K$ is infinitely divisible if 
$$\big(S( z,w)^{-1} K( z, w) \big )^t,$$ 
where 
$$S( z, w):= \prod_{i=1}^m(1-z_i\bar{w}_i)^{-1},\quad  z, w\in \mathbb D^m,$$
is positive definite for all $t>0$.  Every commuting tuple of contractions
$\boldsymbol M^*$ need not be infinitely divisible. Let $\boldsymbol S_m$ be the
commuting $m$ - tuple of the joint weighted shift defined on the
Hardy space $H^2(\mathbb D^m$).

\begin{corollary}(\cite[Corollary 4.4]{misra-BKM})
Let $K$ be a positive definite kernel on the polydisc $\mathbb
D^m$. Assume that the adjoint $\boldsymbol M^*$ of the multiplication
operator $\boldsymbol M$ on the reproducing kernel Hilbert space
$\mathcal H_K$ belongs to $B_1(\mathbb D^m)$. The  function 
$$\big(\!\big (\frac{\partial^2}{\partial w_i\,\partial\bar{w_j}} \log
\big (S(w, w)^{-1} K(w,w) \big )\,\big )\!\big )_{i,j=1}^m,\quad w\in
\mathbb D^m,$$ 
is positive definite, or equivalently
$$ {\mathscr K}_{\boldsymbol M^*}(w) \preceq {\mathscr K}_{\boldsymbol S_m^*}(w),
\quad w\in\mathbb D^m,$$
if and only if the multiplication operator $\boldsymbol M$ is an infinitely
divisible $m$-tuple of row contractions.
\end{corollary}\index{local operator}

Exploiting the explicit description of the local operators $N_1(w), \ldots, N_m(w)$, a very general curvature inequality for a commuting tuple of operators $\boldsymbol T$ in $B_n(\Omega)$ is given in \cite[Theorem 2.4]{misra-MR}. We reproduce below a simple instance of such inequalities taking $\Omega$ to be the unit ball $\mathbb B_m$ in $\mathbb C^m$ and setting $n=1$. 

\begin{theorem}(\cite[Theorem 4.2]{misra-MisPal}) 
Let $\theta_w$ be a bi-holomorphic automorphism of $\mathbb B_m$ such that
$\theta_w(w)=0.$ If $\rho_{\boldsymbol T}$ is a contractive homomorphism
of $\mathcal O(\mathbb B_m)$ induced by the localization $N_{\mathbf
T}(w)$, $\boldsymbol T \in \mathrm B_1(\mathbb B),$ then
$$\mathcal K_{\boldsymbol T}(w)\leq -\overline{D\theta_w(w)}^tD\theta_w(w)=
\tfrac{1}{m+1}\mathcal K_{B}(w),\,\, w\in \mathbb B_m,$$
where  $\mathcal K_B$ is the curvature 
$$- \big (\!\big (\frac{\partial^2}{\partial w_i\,\partial\bar{w_j}}
\log B( w, w) \big )\,\big)\!\big )_{i,j=1}^m$$ 
and 
$$B(z,w) =\big (\tfrac{1}{1-\langle {z},{w}\rangle}\big )^{m+1}$$ 
is  the Bergman kernel of the ball $\mathbb B_m$.
\end{theorem}

\subsection{Homogeneous operators}\index{bounded symmetric domain}
\index{homogeneous operator}\index{matrix unit ball}\index{automorphism group}
Let $\mathcal D$ be a bounded symmetric domain. 
The typical examples are the matrix unit ball $(\mathbb C^{p\times q})_1$ of size $p \times q$, which includes the case of the Euclidean ball, i.e., $q=1$. Let $G:=\mathrm{Aut}(\mathcal D)$ be the bi-holomorphic automorphism group of $\mathcal D$. 

For the matrix unit ball, $G:=\mathrm{SU}(p,q),$ which consists  
of all linear automorphisms leaving the form $\begin{psmallmatrix}I_n & 0\\ 0 &- I_m \end{psmallmatrix}$ on $\mathbb C^{p+q}$ invariant. Thus $g \in  \mathrm{SU}(p,q)$ is of the form 
$\begin{psmallmatrix} a & b \\c & d\end{psmallmatrix}$.  
The group $\mathrm{SU}(p, q)$ acts on $(\mathbb C^{p \times q})_1$ via  the map 
$$g = \begin{pmatrix} a & b\\
c & d \\\end{pmatrix}: z \mapsto (a z + b z) ( c z + d z)^{-1},\,\, z \in 
(\mathbb C^{p \times q})_1.$$
This action is transitive. Indeed $(\mathbb C^{p \times q})_1 \cong \mathrm{SU}(p, q)/ \mathbf K,$ where $\mathbf K$ is the stabilizer of $\mathbf 0$ in $(\mathbb C^{p \times q})_1$.\index{stabilizer}\index{transitive action}

When $\mathcal D$ is a bounded symmetric domain of dimension $m$ and $\mathcal H$ is any Hilbert space, an $m$-tuple $T=(T_1, \ldots, T_m)$ of commuting bounded operators acting on $\mathcal H$ is said to be  homogeneous (cf. 
\cite{misra-MSJOT, misra-twist}) if their joint Taylor spectrum is contained in $\overbar{\mathcal D}$ 
and for every holomorphic automorphism $g$ of $\mathcal D,$ there exists a unitary operator $U_g$ such that \index{Taylor spectrum}
$$
g (T_1, \ldots ,T_m) = (U_g^{-1} T_1 U_g, \ldots , U_g^{-1} T_m U_g).
$$   

\emph{Imprimitivity.} \index{imprimitivity}
More generally, let $G$ be a locally compact second countable (lcsc) topological group and $\mathcal D$ be a lcsc $G$-space.  Suppose that $U: G \to \mathcal U(\mathcal H)$ is a unitary representation of the group $G$ on the  Hilbert space $\mathcal H$ and that 
$\varrho:\mathbf C(\mathcal D) \to \mathcal L(\mathcal H)$ is a $*$ - homomorphism of the $C^*$ - algebra of continuous functions $\mathbf C(\mathcal D)$ on the algebra $\mathcal L(\mathcal H)$ of all bounded operators acting on the Hilbert space $\mathcal H$.
Then the pair $(U , \varrho)$ is said to be a representation of the $G$-space $\mathcal D$  if 
$$\varrho(g \cdot f) = U(g)^* \varrho(f) U(g),\, f\in \mathbf C(\mathcal D),\, g\in G,$$
where $(g \cdot f)(w) = f(g^{-1} \cdot w),\, w\in \mathcal D$. This  is  the  generalization due to Mackey of the imprimitivity relation of  Frobenius.  These are exactly the homogenous commuting tuples of normal operators. 

As before, let $\mathbf K$ be the stabilizer group of $\mathbf 0$ in $G$, thus $G/\mathbf K\cong \mathcal D$, where the identification is obtained via the map: $ g\mathbf K \to  g \mathbf 0$. 
The action of $G$ on $\mathcal D$ is evidently transitive. 
 Given any unitary representation $\sigma$ of $\mathbf K$, one may associate a representation $(U^\sigma, \rho^\sigma)$ of the $G$ -space $\mathcal D$. The correspondence
$$\sigma \to (U^\sigma, \rho^\sigma)$$ \index{semi-simple group}
is an equivalence of categories. \index{parabolic subgroup}
The representation $U^\sigma$ is the representation of $G$  induced by the representation $\sigma$ of the group $\mathbf K$.\index{irreducible representation}
For a semi-simple group $G$, induction from the parabolic subgroups is the key to producing irreducible representations. Along with holomorphic induction, this method gives almost all the irreducible  unitary representations of the semi-simple group $G$. 

We ask what happens if the algebra of continuous functions is replaced by the polynomial ring and the $*$~- ~homomorphism $\varrho$ is required to be merely a homomorphism of this ring. These are the commuting tuples of homogeneous operators. 
\index{imprimitivity relation}

\subsubsection{Quasi-invariant kernels}\index{quasi-invariant kernel}
Let $\mathcal M\subseteq \mathrm{Hol}(\mathcal D)$ be a Hilbert space possessing a reproducing kernel, say, $K$. These holomorphic imprimitivities are exactly homogeneous operators. 
Assume that $\mathcal M$ is a Hilbert module over the polynomial ring $\mathbb C[\mathbf z]$. 
Let $U: G\to \mathcal U(\mathcal M)$ be a unitary representation.
What are the pairs $(U, \rho)$ that satisfy the imprimitivity relation, namely,
$$U_g^*\rho(p) U_g = \rho( p \circ g^{-1}),\,\,g\in G,\,\, p\in \mathbb C[\mathbf z].$$
Suppose that the kernel function $K$ transforms according to the rule 
$$
J_g(z)K(g(z), g(w)) J_g(w)^* = K(z,w),\,\,g \in G,\,\, z,w \in \mathcal D,
$$
for some holomorphic function $J_g:\mathcal D \to \mathbb C$.
Then the kernel $K$ is said to be quasi-invariant, which is equivalent to saying that the map $U_g:f \to J_g\, (f\circ g^{-1}),\, g\in G,$ is unitary.
If we further assume that the $J_g:\mathcal D \to \mathbb C$ is a cocycle, then $U$ is a homomorphism. 
The pair $(U,\rho)$ is a representation of the $G$-space $\mathcal D$  and conversely.

Therefore, our question becomes that of 
 a characterization of all the quasi-invariant kernels defined on $\mathcal D$, 
 or equivalently, finding all the holomorphic cocycles, 
 which is also equivalent to finding all the holomorphic Hermitian homogeneous vector bundles over $\mathcal D$.\index{holomorphic cocycle}

Let $K:\Omega\times \Omega \to  \mathcal M_m(\mathbb C)$ be a kernel function.  We will assume the function $K$ is holomorphic in the first variable and anti-holomorphic in the second. For two functions of the form $K(\cdot, w_i)\zeta_i$, $\zeta_i$ in $\mathbb C^m$ ($i=1,2$) define their inner product by the reproducing property, that is, 
$$\langle K(\cdot,w_1) \zeta_1, K(\cdot, w_2) \zeta_2\rangle = \langle K(w_2,w_1) \zeta_1, \zeta_2\rangle.$$ 
This extends to an inner product on the linear span if and only if $K$ is positive definite in the sense that  
\begin{eqnarray*}
\sum_{j,k=1}^n\langle K(z_j,z_k) \zeta_{k}, \zeta_j \rangle &=& \Big \langle \sum_{k=1}^n  K(\cdot, z_k) \zeta_k , \sum_{j=1}^n K(\cdot, z_j)\zeta_j \Big \rangle\\
&=&\Big \|\sum_{k=1}^n \langle K(\cdot, z_k) \zeta_k \Big \|^2 \geq0.
\end{eqnarray*}

Let $G:\Omega\times \Omega \to \mathcal M_m(\mathbb C)$ be the Grammian
$G(z,w) = \left( \langle e_p(w), e_q(z) \rangle \right )_{p,q}$
of a set of $m$ antiholomorphic functions $e_p:\Omega \to \mathcal H$, $1\leq p \leq m$, taking values in some Hilbert space $\mathcal H$. Then\index{Grammian}
\begin{eqnarray*}
\sum_{j,k=1}^n \langle G(z_j,z_k) \zeta_k, \zeta_j\rangle &=& \sum_{j,k=1}^n \big ( \sum_{pq=1}^m  \langle e_p(z_k),e_q(z_j) \rangle \zeta_k(p) \overline{\zeta_j(q)}\,\big ) \\
&=& \Big \| \sum_{j,k}\zeta_k(p) e_p(z_k) \Big \|^2 > 0.
\end{eqnarray*}
We therefore conclude that $G(z,w)^{\rm tr}$ defines a positive definite kernel on $\Omega$.  

For  an anti-holomorphic function $s:\Omega \to \mathbb C^m$, let us define the norm, at $w$, $\|s(w)\|^2:=\| K(\cdot, w) s(w) \|^2$, where the norm on the right hand side is the norm of the Hilbert space $\mathcal H$ defined by the positive definite kernel $K$. Let $\varepsilon_i$, $1 \leq i \leq m$,  be the standard basis vectors in $\mathbb C^m$. We have 
\begin{eqnarray*}
\Big \| K(\cdot, w) s(w) \Big \|^2 &=& \Big \|\sum_i s_i(w) K(\cdot, w) \varepsilon_i \Big \|^2\\
&=& \sum_{i,j} \langle  s_i(w) K(\cdot, w) \varepsilon_i , s_j(w) K(\cdot, w) \varepsilon_j \rangle\\
&=& \sum_{i,j} s_i(w) \overline{s_j(w)} \langle  K(\cdot, w) \varepsilon_i , K(\cdot, w) \varepsilon_j \rangle\\
&=& \sum_{i,j} K(w,w)_{j,i}\, s_i(w) \overline{s_j(w)}\\
&=& \overline{\big (s(w)\big )}^{\rm \,tr}K(w,w)^{\rm tr} s(w)
\end{eqnarray*}
For $w$ in $\Omega$, and $p$ in the set $\{1,\ldots,m\}$, let $e_p:\Omega \to \mathcal H$ be the antiholomorphic function: 
$$e_p(w):=K_w(\cdot)\otimes\frac{\partial}{\partial \bar{w}_p}K_w(\cdot) - 
\frac{\partial}{\partial \bar{w}_p}K_w(\cdot)\otimes K_w(\cdot).$$
Then 
$$
\frac{1}{2}\langle e_p(w), e_q(z) \rangle =  K(z,w)\frac{\partial^2}{\partial z_q\partial \bar{w}_p}K(z,w) - \frac{\partial}{\partial \bar{w}_p}K(z,w) \frac{\partial}{\partial {z}_q}K(z,w).
$$
The curvature of the metric $K$ is given by the $(1,1)$ - form  
$$\sum \frac{\partial^2}{\partial w_q\partial \bar{w}_p}\log K(w,w)\,dw_q
\wedge d\bar{w}_p.$$  
Set 
$$\mathcal K(z,w) :=  \left(\frac{\partial^2}{\partial z_q\partial \bar{w}_p}\log K(z,w) 
\right)_{qp}.$$
Since $2 K(z,w)^2 \mathcal K(z,w)$ is of the form $G(z,w)^{\rm tr}$, it follows that $K(z,w)^2 \mathcal K(z,w)$ defines a positive definite kernel on $\Omega$ taking values in $\mathcal M_m(\mathbb C)$.

Let $\varphi:\Omega \to \Omega$ be holomorphic and $g:\Omega \to \mathbb C$ be the function $g(w):=K(w,w)$, $w\in \Omega$. Set $h = g \circ \varphi$. Apply the change of variable formula twice. The first time around, we have     
$$ 
(\partial_i h)(w) = \sum_\ell (\partial_\ell g)(\varphi(w)) (\partial_i\varphi_\ell)(w),$$      
differentiating a second time, we have
$$
(\bar{\partial}_j \partial_i h)(w) = \sum_\ell \Big (\sum_k \bar{\partial}_k \big ( \partial_\ell g \big ) (\varphi(w)) \overline{\partial_j \varphi_k(w)}\Big ) (\partial_i\varphi_\ell)(w).
$$
In terms of matrices, we have 
\begin{align*}
\left(\frac{\partial^2}{\partial z_i \partial \bar{w}_j}\log K(\varphi(z),\varphi(w)) 
\right)_{i,j} 
&=\left(\frac{\partial\varphi_\ell}{\partial z_i} \right)_{i,\ell} 
\left( \Big (\frac{\partial^2}{\partial z_\ell\partial \bar{w}_k}\log K\Big )\big (\varphi(z),\varphi(w) \big )
\right)_{\ell, k} \left(\frac{\partial\bar{\varphi}_k}{\partial \bar{z}_j} \right)_{k,j}.
\end{align*}
Equivalently, 
$$
\mathcal K_{g\circ \varphi}(w) = D\varphi(w)^{\rm tr} \mathcal K_g(\varphi(w)) \overline{D\varphi(w)}.
$$
If 
$${\det}_{\mathbb C}\, D\varphi (w)\, \big (g\circ \varphi \big )(w) \,\overline{{\det}_{\mathbb C} 
D\varphi (w)} = g(w),$$ 
then $\mathcal K_{g\circ \varphi}(w)$ equals $\mathcal K_g(w)$. Hence we conclude that 
$\mathcal K$ is invariant under the automorphisms $\varphi$ of $\Omega$ in the sense that 
$$
D\varphi(w)^{\rm tr} \mathcal K(\varphi(w)) \overline{D\varphi(w)} = \mathcal K(w),\,\, w\in \Omega.
$$   
Let $Q: \Omega \to \mathcal M_m(\mathbb C)$ be a real analytic function such that $Q(w)$ is  positive definite for $w\in \Omega$.  Let $\mathcal H$ be the Hilbert space of holomorphic functions on $\Omega$ which are square integrable with respect to $Q(w) dV(w)$, that is, 
$$
\|f\|^2:= \int_\Omega \langle Q(w)f(w), f(w) \rangle^2 dV(w),
$$
where $dV$ is the normalized volume measure on $\mathbb C^m$. Let $U_\varphi:\mathcal H \to \mathcal H$ be the operator 
$$(U_\varphi f)(z) = m(\varphi^{-1}, z) (f\circ \varphi^{-1})(z)$$ 
for some cocycle $m$. The operator $U_\varphi$ is unitary if and only if 
\begin{eqnarray*}
\|U_\varphi f\|^2 &=& \int \langle Q(w) (U_\varphi f)(w), (U_\varphi f)(w) \rangle^2 dV(w)\\   
&=& \int \langle \overline{m(\varphi^{-1},w)}^{\rm \,tr} Q(w) m(\varphi^{-1},w) f(\varphi^{-1}(w)), f(\varphi^{-1}(w))  \rangle^2 dV(w)\\
&=& \int  \langle Q(w) f(w), f(w) \rangle^2 dV(w),
\end{eqnarray*}
whenever $Q$ transforms according to the rule
\begin{equation}\label{Q}
\overline{m(\varphi^{-1},w)}^{\rm \,tr} Q(w) m(\varphi^{-1},w) = Q(\varphi^{-1}(w)) | {\det}_{\mathbb C} (D\varphi^{-1})(w)|^2.
\end{equation}
Set 
$$m(\varphi^{-1},w)= D\varphi^{-1}(w)^{\rm tr}$$ 
and 
$$Q^{(\lambda)}(w):=b(w)^{1-\lambda}\mathcal K_g(w)^{-1},\quad\lambda > 0,$$ 
where  $b(w)$ is the restriction of the Bergman kernel to the diagonal subset of $\Omega \times \Omega$. Then $Q^{(\lambda)}$ transforms according to the rule \eqref{Q}. If for some $\lambda > 0$, the Hilbert space $L_{\rm hol}^2(\Omega, Q^{(\lambda)}\,dV)$ determined by the measure is nontrivial, then the corresponding reproducing kernel is of the form $b(w)^\lambda \mathcal K(w)$. 

For the Euclidean ball, 
$$L_{\rm hol}^2(\mathbb B^m, Q^{(\lambda)}\,dV)\not = \{0\}$$ 
if and only if $\lambda > 1$. This means that the polarization of the real analytic function 
\begin{align*}
(1-\langle w,w \rangle ) (1-\langle w,w \rangle )^{-\lambda (n+1)} \mathcal K(w) &= (1-\langle w,w \rangle )^{-(\lambda - 1)(n+1)} \mathcal K(w)\\
&= (1-\langle w,w \rangle )^{-(\lambda + 1) (n+1)} \left( \langle e_p(w), e_q(w) \rangle 
\right)_{qp}
\end{align*} 
must also be positive definite. Since $\left( \langle e_p(w), e_q(z) \rangle \right )_{qp}$ is positive definite, it follows that $\lambda > -1$ will ensure positivity of the kernel 
$$
(1-\langle w,w \rangle )^{-\lambda (n+1)} \mathcal K(w) = b(w)^{\lambda + 2}  
{\rm Adj} (I_n - z w^*).
$$

\subsubsection{Classification} 
We have already described homogeneous operators in the Cowen-Douglas class $B_1(\mathbb D)$ using the curvature invariant. For operators in  $B_k(\mathbb D)$, $k >1$, the curvature alone does not determine the class of the operator.  Examples of irreducible homogeneous operators in $B_k(\mathbb D)$ were given in 
\cite{misra-KM2008} using an intertwining operator $\Gamma$.  For the complete classification, we recall the description of homogeneous vector bundles via holomorphic induction, see \cite{misra-Kir}. Making this explicit in our context, we were able to construct the intertwining operator $\Gamma$ in general \cite{misra-KM2011}.  Some of the details are reproduced  below from the  announcement \cite{misra-KM2009}. \index{homogeneous vector bundle}\index{intertwining operator}

Let $\mathfrak t \subseteq {\mathfrak g}^{\mathbb C}=\mathfrak s\mathfrak l(2,\mathbb C)$ be the algebra $ \mathbb C h + \mathbb C y$, where
$$ h = \frac{1}{2}\begin{pmatrix} 1 & 0\\ 0 & - 1 \end{pmatrix},\quad 
y = \begin{pmatrix} 0 & 0\\ 1 & 0 \end{pmatrix}.$$
Linear representations $(\varrho, V)$ of the algebra $\mathfrak t \subseteq 
{\mathfrak g}^{\mathbb C}=\mathfrak s\mathfrak l(2,\mathbb C)$, that is, pairs $\varrho(h), \varrho(y)$ of linear transformations satisfying $[\varrho(h),\varrho(y)] = - \varrho(y)$ provide a para-metrization of the homogeneous holomorphic vector bundles.

\index{universal covering group}\index{representation space}\index{stabilizer}
In obtaining the classification of the homogeneous operators, it is necessary to work with the universal covering group $\tilde{G}$ of the bi-holomorphic automorphism group $G$ of the unit disc. The $\tilde{G}$ - invariant Hermitian structures on the homogeneous holomorphic vector bundle $E$ (making it into a homogeneous holomorphic Hermitian vector bundle), if they exist, are given by $\varrho(\tilde{\mathbb K})$ - invariant inner products on the representation space.  Here $\tilde{\mathbb K}$ is the stabilizer of $0$ in $\tilde{G}$.

An inner product can  be $\varrho(\tilde{\mathbb K})$ - invariant if and only if $ \varrho(h)$ is diagonal with real diagonal elements in an appropriate basis.  We are interested only in Hermitizable bundles, that is, those that admit a Hermitian structure.  So, we will assume without restricting generality, that the representation space of $\varrho$ is $\mathbb C^n$ and that $\varrho(h)$ is a real diagonal matrix.

Since $[\varrho(h),\varrho(y)] = - \varrho(y)$, we have $\varrho(y) V_\lambda \subseteq V_{\lambda -1}$, where $V_\lambda = \{ \xi \in \mathbb C^n: \varrho(h) \xi = \lambda \xi\}$.
Hence $(\varrho, \mathbb C^n)$ is a  direct sum, orthogonal for every $\varrho(\tilde{K})$ - invariant inner   product of ``elementary'' representations,  that is, such that
$$ \varrho(h) = \begin{pmatrix}
-\eta I_0 & &\\
 &\ddots & \\
&&-(\eta + m) I_m
\end{pmatrix}\mbox{with}\,\, I_j = I\,\, \mbox{on}\,\, V_{-(\eta+j)} = \mathbb C^{d_j}$$
and
$$ Y:= \varrho(y) = \begin{pmatrix}
0 & & & \\
Y_1 & 0 & &\\
&Y_2&0 & &\\
&&\ddots&\ddots&\\
&&&Y_m&0
\end{pmatrix},\quad Y_j:V_{-(\eta+j-1)} \to V_{-(\eta+j)}.$$
We denote the corresponding elementary Hermitizable bundle by $E^{(\eta, Y)}$. 

\emph{The Multiplier and Hermitian  structures.}
As in \cite{misra-KM2011} we will use a natural trivialization of $E^{(\eta, Y)}$. In  this,  the sections of homogeneous holomorphic vector bundle $E^{(\eta, Y)}$ are holomorphic functions $\mathbb D$ taking values in $\mathbb C^n$.  The $\tilde{G}$ action is given by $f  \mapsto J_{g^{-1}}^{(\eta,  Y)} \big (f\circ g^{-1}\big )$ with multiplier
$$\left( J^{(\eta, Y)}_g(z))\right)_{p,\ell} =\begin{cases}
\tfrac{1}{(p-\ell)!} (-c_g)^{p-\ell}
(g^\prime)(z)^{\eta + \frac{p + \ell}{2}} Y_p \cdots Y_{\ell+1} &{~if~} p \geq \ell\\
0&{~if~} p < \ell \end{cases},
$$
where $c_g$ is  the analytic function on $\tilde{G}$ which, for $g$ near $e$, acting on $\mathbb D$ by $z \mapsto (a z + b) (cz +d)^{-1}$ agrees with $c$.
\begin{proposition}
We have $E^{(\eta, Y)}\equiv E^{(\eta^\prime, Y^\prime)}$ if and only if $\eta=\eta^\prime$ and $Y^\prime = AYA^{-1}$ with a block diagonal matrix $A$.
\end{proposition}
A Hermitian structure on $E^{(\eta, Y)}$ appears as the assignment of an inner product  $\langle \cdot , \cdot \rangle_z$ on  $\mathbb C^n$ for   $z \in \mathbb D$. We can write 
$$ \langle \zeta, \xi \rangle_z = \langle H(z) \zeta , \xi \rangle,\,\,\mbox{with}\,\, H(z) \succ 0.$$
Homogeneity as a Hermitian vector bundle is equivalent to
$$ J_g(z) H(g\cdot z)^{-1}J_g(z)^* = H(z)^{-1},\,\,g \in G,\,\, z\in \mathbb D.$$
The Hermitian   structure is then determined by $H=H(0)$ which is a positive block diagonal matrix.  We write $(E^{(\eta, Y)}, H)$ for the vector bundle $E^{(\eta, Y)}$ equipped with the
Hermitian structure $H$.  We note that $(E^{(\eta, Y)}, H) \cong (E^{(\eta, AYA^{-1})}, {A^*}^{-1}HA)$ for any block diagonal invertible $A$.  Therefore every homogeneous holomorphic Hermitian vector bundle is isomorphic with one of the form  $(E^{(\eta, Y)}, I)$.\index{block diagonal matrix}

If $E^{(\eta, Y)}$ has a reproducing kernel $K$ which
is the case for bundles corresponding to an operator in the  Cowen-Douglas class, 
then $K$ satisfies 
$$K(z,w) = J_g(z) K(g z , g w) J_g(w)^*$$
and induces a Hermitian structure $H$ given by $H(0) = K(0,0)^{-1}.$
\emph{Construction of the bundles with reproducing kernel.}
For $\lambda >0$, let $\mathbb A^{(\lambda)}$ be the Hilbert space of holomorphic functions on the unit disc with reproducing kernel $(1-z \bar{w})^{-2 \lambda}$.
It is  homogeneous under the multiplier  $g^\lambda(z)$ for the action of $\tilde{G}$. This gives a unitary representation of $\tilde{G}$.  Let  
$$\mathbf A^{(\eta)}= \oplus_{j=0}^m \mathbb A^{(\eta + j)} \otimes \mathbb C^{d_j}.$$  
For $f$ in $\mathbf A^{(\eta)}$, we denote by $f_j$, the part of $f$ in $\mathbb A^{(\eta + j)} \otimes \mathbb C^{d_j}$.  We define $\Gamma^{(\eta, Y)} f$ as the $\mathbb C^n$ - valued holomorphic function whose part in $\mathbb C^{d_\ell}$ is given by
$$ \big (\Gamma^{(\eta, Y)} f \big )_\ell = \sum_{j=0}^\ell\frac{1}{(\ell -j)!}\frac{1}{(2 \eta + 2 j)_{\ell -j}} Y_\ell \cdots Y_{j+1} f_j^{(\ell -j)}$$
for $\ell \geq j$. For invertible block diagonal $N$ on $\mathbb C^n$, we also define 
$\Gamma_N^{(\eta, Y)}:= \Gamma^{(\eta, Y)} \circ N$.  It can be verified that $ \Gamma_N^{(\eta, Y)}$ is a $\tilde{G}$ -   equivariant isomorphism of $\mathbf A^{(\eta)}$ as a homogeneous holomorphic vector bundle onto $E^{(\eta, Y)}$.  The image $K_N^{(\eta, Y)}$ of the reproducing kernel of $\mathbf A^{(\eta)}$ is then a reproducing kernel for $E^{(\eta, Y)}$.  A computation gives that $K^{(\eta,Y)}_N(0,0)$ is a block diagonal matrix such that its $\ell$'th block is
$$
K^{(\eta, Y)}_N(0,0)_{\ell,\ell} = \sum_{j=0}^\ell \frac{1}{(\ell -j)!}\frac{1}{(2\eta + 2j)_{\ell-j}} Y_\ell\cdots Y_{j+1} N_jN_j^* Y_{j+1}^* \cdots Y_\ell^*.
$$
We set 
$$H^{(\eta,Y)}_N = {K_N^{(\eta, Y)}(0,0)}^{-1}.$$  
We have now constructed a family $(E^{(\eta, Y)}, H^{(\eta,Y)}_N)$ of elementary homogeneous holomorphic vector bundles with a reproducing kernel  ($\eta >0$, $Y$ as before, $N$ invertible block diagonal).

\begin{theorem} (\cite[Theorem 3.2]{misra-KM2011}
Every elementary homogeneous holomorphic vector bundle $E$ with a reproducing kernel arises
from the construction given above.
\end{theorem}

\begin{proof} ({\bf Sketch of proof})
As a homogeneous bundle $E$ is isomorphic to some $E^{(\eta, Y)}$.   Its reproducing kernel gives a Hilbert space structure in which the $\tilde{G}$ action on the sections of $E^{(\eta, Y)}$ is a unitary representation $U$.  Now $\Gamma^{(\eta, Y)}$ intertwines the unitary representation of $\tilde{G}$ on $\mathbb A^{(\eta)}$ with $U$.  The existence of a block diagonal $N$ such that $\Gamma_N^{(\eta, Y)} = \Gamma^{(\eta, Y)} \circ N$ is a Hilbert space isometry follows from Schur's Lemma.\index{Schur's Lemma}
\end{proof}

As remarked before, every homogeneous holomorphic Hermitian vector bundle is isomorphic to an $(E^{(\eta,Y)}, I)$, here $Y$ is unique up to conjugation by a block unitary.  In this form, it is easy to tell whether the bundle is irreducible:  this is the case if and only if $Y$ is not the orthogonal direct sum of two matrices of the same block type as $Y$. We call such a $Y$ irreducible.

Let $\mathcal P$ be the set of all $(\eta, Y)$ such that $E^{(\eta, Y)}$ has a reproducing kernel.  Using the formula for $K_N^{(\eta, Y)}(0,0)$ we can write down explicit systems of inequalities that determine whether $(\eta, Y)$ is in $\mathcal P$.    In particular we have
\begin{proposition}
For every $Y$, there exists a $\eta_Y > 0$ such that $(\eta, Y)$ is in $\mathcal P$ if and only if $\eta>\eta_Y $.
\end{proposition}

Finally, we obtain the desired classification.

\begin{theorem}(\cite[Theorem 4.1]{misra-KM2011})
All the homogeneous holomorphic Hermitian vector bundles of rank $n$ with a reproducing kernel correspond to homogeneous operators in the Cowen -- Douglas class ${\rm B}_n(\mathbb D)$. The irreducible ones are the adjoint of the multiplication operator $M$ on  the Hilbert space of sections of $(E^{(\eta, Y)},I)$ for some $(\eta, Y)$ in $\mathcal P$ and irreducible $Y$.  The block matrix $Y$ is determined up to conjugacy by block diagonal unitaries.
\end{theorem}

\begin{proof} ({\bf Sketch of proof})
There is a simple orthonormal system for the Hilbert space  $\mathbb A^{(\lambda)}$.  Hence we can find such a system for $\mathbf A^{(\eta)}$ as well.  Transplant it using  
$\Gamma^{(\eta, Y)}$ to   $E^{(\eta,  Y)}$. The multiplication operator in this basis has 
a block diagonal form with   
$$M_n:= M_{|\,{\rm res\,\,} \mathcal H(n)}: \mathcal H(n) \to \mathcal H(n+1).$$ 
This description is sufficiently explicit to see: 
$$ M_n \sim I + 0(\frac{1}{n}).$$ 
Hence $M$ is the sum of an ordinary  block shift operator and a Hilbert Schmidt 
operator. This completes the proof.\index{block shift operator}
\index{Hilbert-Schmidt operator}
\end{proof}

\emph{The general case.} \index{Wilkins operator}
The first examples of homogeneous operators in several variables were given in 
\cite{misra-MSJOT}. This was followed by a detailed study of these operators in the paper \cite{misra-twist} for tube domains and in full generality in the paper \cite{misra-AZ}. First examples of homogeneous operators in $B_2(\mathbb D)$ were given in \cite{misra-DW}. A class of homogeneous operators in $B_n(\mathbb D)$, which we called, generalized Wilkins operators were described in \cite{misra-BMsurvey} using the \emph{jet construction}.  The paper \cite{misra-PZ} gives a class of homogeneous operators in $B_n(\mathcal D)$, where $\mathcal D$ is one of the classical bounded symmetric domains, using the decomposition of a tensor product of two discrete series representations.  

The essential ingredients from the case of the automorphism group of the unit disc is now available for an arbitrary bounded symmetric domain. In particular, the intertwining operator $\Gamma$ has been found explicitly. This gives a complete classification of homogeneous operators in the Cowen-Douglas class of the ball.  In general, while a description of the homogeneous holomorphic vector bundles is given in the paper \cite{misra-KM2018}, it hasn't been possible to describe the operators as explicitly as in the case of the ball.  
These results were announced in \cite{misra-KM2016} and complete proofs now appear in \cite{misra-KM2018}.  A different approach to finding a class of homogeneous holomorphic vector bundles is in \cite{misra-MU}.  

As one might expect several questions remain open, for instance, which commuting tuple of homogeneous operators are subnormal, which of them induce a contractive, or completely contractive homomorphism.  Some of these questions have been studied in 
\cite{misra-twist, misra-AZ}. In a different direction, the class of quasi-homogeneous operators introduced in \cite{misra-JJM}  containing all the homogeneous operators shares many of the properties of the smaller class of homogeneous operator.  The Halmos question on similarity of polynomially bounded operators to a contraction has an affirmative answer for the quasi-homogeneous operators. 

\subsection{Quotient and sub-modules} 
In an attempt to generalize the very successful model theory of Sz.-Nagy and Foias for contractions to other settings, R. G. Douglas reformulated the question in the language of Hilbert modules over a function algebra. We describe below some aspects of this reformulation and its consequences \index{quasi-free Hilbert module}
focussing on the quotient modules for the class of quasi-free Hilbert modules $\mathcal M$  introduced in \cite{misra-quasi}. These Hilbert modules are obtained by taking the completion of the polynomial ring $\mathbb C \mathbb[\boldsymbol z]$ with respect to some inner product.  We consider a class of sub-modules $\mathcal S_k \subseteq \mathcal M$ which consist of all functions in $\mathcal M$ that vanish to some fixed order $k$ on a hypersurface $\mathcal Z$ contained in $\Omega$. Let us recall some of the definitions (cf. \cite[Section 1.2]{misra-qmtrans}).

\begin{enumerate} 
\item[(1)] A \emph{ hypersurface} is a complex sub-manifold of complex dimension
$m-1$, that is, a subset $\mathcal Z \subseteq \Omega$ is a hypersurface
if for any fixed $z \in {\mathcal Z}$, there exists a
neighbourhood $U \subseteq \Omega$ of $z$ and a local defining
function $\varphi$ for $U\cap {\mathcal Z}$. 

\item[(2)] A \emph{ local defining function} $\varphi$ is a holomorphic map $\varphi: U
\to \mathbb{C}$ such that $U \cap {\mathcal Z} = \{z \in U: \varphi(z) = 0 \}$
and $\tfrac{f}{\varphi}$ is holomorphic on $U$ whenever $f_{| U\cap
  \mathcal Z} =0$. In particular, this implies that the gradient of
$\varphi$ doesn't vanish on $\mathcal Z$ and that any two defining
functions for $\mathcal Z$ must differ by a unit.

\item[(3)] A function $f$ is said to be \emph{ vanishing to order $k$} on the hypersurface
$\mathcal Z$ if $f = \varphi^n g$ for some $n\geq k$, a holomorphic function
$g$ on $U$ and a defining function $\varphi$ of $\mathcal Z$.  The order of
vanishing on $\mathcal Z$ of a holomorphic function $f: \Omega \to \mathbb C$ does
not depend on the choice of the local defining function.\index{hypersurface}
This definition can also be framed in terms of the partial derivatives 
normal to $\mathcal Z$.\index{local defining function}
\end{enumerate}
We have seen that an extension of a result due to Aronszajn's provides a model for the quotient module when the sub-module consist of the maximal set of all functions vanishing on a hyper-surface. However, if the sub-module is taken to be all functions vanishing to order $k > 1$, then the situation is different and one must introduce a matrix valued kernel via the jet construction.\index{jet construction}

\subsubsection{The jet construction} 
Let $\mathcal S_k \subseteq \mathcal M$ be a sub-module of a quasi-free Hilbert module over the algebra $\mathcal A(\Omega)$ consisting of functions vanishing to order $k$ on a hyper-surface $\mathcal Z\subseteq \Omega$. Let $\partial$ denote the differentiation along the unit normal to the hyper-surface $\mathcal{Z}$.  Let $J: \mathcal{M}
\to \mathcal {M} \otimes \mathbb{C}^k$ defined by 
$$
h \mapsto (h, \partial h, \partial^2 h, \ldots , \partial^{k-1} h), ~
h\in \mathcal{M}
$$  
be the jet of order $k$. 
Transplanting the inner product from $\mathcal M$ on $\mathcal {M} \otimes \mathbb{C}^k$ via the map $J$, we see that 
$$\{(e_n, \partial e_n, \ldots , \partial^{k-1}e_n)_{n \geq 0}:
(e_n)_{n\geq 0} \mbox {~is an orthonormal basis in $\mathcal{M}$~}\}$$ is
an orthonormal basis in $J(\mathcal{M}) \subseteq \mathcal{M}\otimes \mathbb{C}^k$. This makes the jet map $J$ isometric.  Now, it is not hard to see that  the sub-module $J\mathcal S_k \subseteq J\mathcal M$  consisting of the maximal set of functions in $J\mathcal M$ vanishing on $\mathcal Z$ is exactly the image under the map $J$ of the sub-module $\mathcal S_k \subseteq \mathcal M$. It is shown in \cite{misra-DMV} that the quotient module 
$\mathcal{Q}= {\mathcal M}/{\mathcal S_k}$ is isomorphic to  
${(J\mathcal{M})}/{(J\mathcal{S}_k)}$.  Thus we are reduced
to the multiplicity free case  and it follows that the quotient module $\mathcal{Q}$ is  
the restriction of $J\mathcal{M}$ to the hypersurface $\mathcal{Z}$. To complete the description, we must provide a model. 

The module $J\mathcal{M}$ possesses a reproducing kernel $JK$, which is the infinite sum 
$$
JK(z,w) = \sum_{n=0}^\infty (Je_n)(z) (Je_n)(w)^*,~z,w \in \Omega.
$$
From this it follows that $JK:\Omega \times \Omega \to \mathcal M_k(\mathbb C)$ is of the form 
\begin{equation}\label{jet kernel}
(JK)_{\ell,j}(z,w) =\big (\partial^{\ell}\bar{\partial}^{j}K \big )
(z,w),\,\, 0\leq \ell,j \leq k-1.
\end{equation}
The module multiplication on $J\mathcal{M}$ is then naturally obtained by requiring that the map $J$ be a module map. Thus setting $Jf$ to be the array
\begin{equation} \label{jet action}
(Jf)_{\ell,j} = \begin{cases}  \binom{\ell}{j} (\partial^{\ell-j}f),& 0\leq
\ell\leq j \leq k-1\\ 0 & \mbox{otherwise,} \end{cases}, 
\end{equation}
$f \in \mathcal{A}(\Omega)$.  The module multiplication is induced by the multiplication operator $M_{J_f}$, where 
$J_f$ is a holomorphic function defined on $\Omega$ taking values in $\mathcal M_k(\mathbb C)$. 
The adjoint of this operator is easy to compute 
\begin{equation} \label{adjoint action}
J_f^* JK(\cdot,w) \cdot \boldsymbol x = JK(\cdot,w)  (Jf)(w)^*\cdot \boldsymbol x, ~\boldsymbol x \in \mathbb{C}^k.
\end{equation} 

We consider the Hilbert space $(J{\mathcal M})_{\rm res}$ obtained by
restricting the functions in $J{\mathcal M}$ to the hyper-surface ${\mathcal Z}$,
that is,
$$
(J{\mathcal M})_{\rm res}=\left\{ \boldsymbol h_0 \mbox{\rm~holomorphic on~} {\mathcal Z}:  
\boldsymbol h_0=\boldsymbol h_{|\,\mathcal Z} \mbox{ for some } \boldsymbol h\in J{\mathcal M}\right\}.   
$$
The norm of $\boldsymbol h_0$ in $(J{\mathcal M})_{\rm res}$ is the quotient norm, that is, 
$$
\|\boldsymbol h_0\|=\inf\, \big \{\|\boldsymbol h\|:\boldsymbol h_{|\,\mathcal Z} = \boldsymbol h_0 \mbox{ for } \boldsymbol h \in J{\mathcal M} \big \},
$$
and the module action is obtained by restricting the map $(f, \boldsymbol h) \to (J_f)_{|\,\mathcal Z} \,\boldsymbol h_{|\,\mathcal Z}$. 

We have discussed the restriction map $R$ before and observed that it is a unitary module map in the case of Hilbert modules consisting of scalar valued
holomorphic functions. However, this is true of the vector valued case as well. It shows that$JK(\cdot,w)_{\rm res} := K(\cdot ,w)_{|\,\mathcal Z}$, $w\in {\mathcal Z}$ is the kernel function for the Hilbert module $(J{\mathcal M})_{\rm res}$. 
\begin{theorem}(\cite[Theorem 3.3]{misra-DMV}) 
Let ${\mathcal M}$ be a Hilbert module over the algebra ${\mathcal
A}(\Omega)$. If ${\mathcal S}_k$ is the sub-module of functions vanishing to order $k$, then the quotient module is isomorphic to $(J{\mathcal M})_{\rm res}$  
via the isometric module map $J\,R$. 
\end{theorem}

Finding invariants for the quotient module, except in the case of $k=1$, is more challenging. The module multiplication in the quotient module involves both a semi-simple and a nilpotent part. The semi-simple part lies typically in some $B_n(\mathcal Z)$. Now, any equivalence between two quotient modules must also intertwine the nilpotent action.  In the papers \cite{misra-PIASeqHM, misra-qmtrans}, using this additional structure, a complete invariants were found for a class of quotient modules.    We describe some fascinating possibilities for finding invariant for quotient modules using the notion of module tensor products and the recent work of Harvey and Lawson.

Let ${\mathcal M}$ and ${\mathcal N}$ be any two
Hilbert modules over the algebra ${\mathcal A}$.  Notice that there are
two possible module actions on ${\mathcal M}\otimes {\mathcal N}$, i.e., the
left action: $L \otimes I: (f,h\otimes k) \mapsto f\cdot h \otimes k$
and the right action: $I \otimes R: (f,h\otimes k) \mapsto h \otimes f
\cdot k$.  The module tensor product ${\mathcal M} \otimes_{\mathcal A} {\mathcal
N}$ is defined to be the module obtained by identifying these two
actions. Specifically, let ${\mathcal S}$ be the closed subspace of ${\mathcal
M} \otimes {\mathcal N}$ generated by vectors of the form
$$\{f\cdot h \otimes k - h\otimes f\cdot k : h\in {\mathcal M}, k\in {\mathcal
N}\mbox{ and } f\in {\mathcal A}\}.$$ Then ${\mathcal S}$ is a submodule of
${\mathcal M} \otimes {\mathcal N}$ with respect to both the left and the
right actions.  The module tensor product ${\mathcal M}\otimes_{\mathcal A}
{\mathcal N}$ is defined to be ${(\mathcal M \otimes \mathcal N)}/
{\mathcal S}$ together with the compression of either the left or the
right actions, which coincide on this space.  For fixed $w \in
\Omega$, ${\mathbb C}$ is a module over ${\mathcal A}$ with the module action
$$ (f,v) \mapsto f(w) v, ~ f\in {\mathcal A}, v\in {\mathbb C}.$$ Let ${\mathbb 
C}_w$ denote the one dimensional module ${\Bbb C}$ with this action.
We will largely confine ourselves to the module tensor product ${\mathcal
M} \otimes_{\mathcal A} {\mathbb C}_w$, which we denote by ${\mathcal M}(w)$.  

Localizing the short exact sequence 
$$0 \longleftarrow  {\mathcal Q} \longleftarrow
{\mathcal M} \stackrel{X}{\longleftarrow} {\mathcal S} \longleftarrow 0$$
using the one dimensional module $\mathbb C_w$, one obtains a new  short 
exact sequence of spectral sheaves (cf. \cite{misra-DP})  
$$ 0 \longleftarrow \mathcal Q\otimes_{\mathcal A(\Omega)} \mathbb{C}_w
\longleftarrow \mathcal M\otimes_{\mathcal A(\Omega)} \mathbb{C}_w\stackrel{X
\otimes_{\mathcal A(\Omega)}1_{\mathbb C_w}}{\longleftarrow} \mathcal
S \otimes_{\mathcal A(\Omega)} \mathbb{C}_w \longleftarrow 0.
$$
Let $E_0$ and $E$ be the holomorphic line bundles corresponding to the modules $\mathcal S$ 
and $\mathcal{M}$ and $\mathcal K_{E_0}$, $\mathcal K_E$ be their  curvatures, respectively.  
It is shown in \cite{misra-rgdgm} that the class of the alternating sum 
$$
\sum_{i,j=1}^m \frac{\partial^2}{\partial w_{i}
\partial\bar{w}_{j}} \log (X(w)^*X(w)) dw_i \wedge d\bar{w}_j
- {\mathcal K}_{E_0}(w) + {\mathcal K}_E(w), 
$$ 
is the fundamental class of the hypersurface $\mathcal Z$. This
identification makes essential use of the Poincar\'{e}-Lelong formula.
\index{Poincar\'{e}-Lelong formula}

It is not clear how one can obtain such an alternating sum if the
submodule $\mathcal{S}$ is not assumed to be the submodule (maximal set) of functions
vanshing on $\mathcal{Z}$. 
A possible approach to this question using some ideas of Doanldson appearing in \cite{misra-Don} is discussed in \cite{misra-DMV}.
An  adaptation of the results of Harvey and Lawson \cite{misra-HL} to
the present situation may be fruitful in the case where $\mathcal{S}$ is 
assumed to consist of all functions in $\mathcal{M}$ which vanish to order $k$.
Let ${\mathcal E}$ and ${\mathcal E}_0$ be the vector bundles obtained
by localization (cf. \cite{misra-DP}) from the modules ${\mathcal M}$ and
${\mathcal M}_0$ and let $\phi$ be an ad-invariant polynomial (in
particular, a Chern form) in the respective curvatures ${\mathcal K}$
and ${\mathcal K}_0$.  Then the work of Harvey and Lawson \cite{misra-HL} on
singular connections gives a mechanism for studying these bundles
since the natural connection on the bundle ${\mathcal M}_0$ is a
singular one.  They obtain a relation of the form
$$
\phi({\mathcal K}) - \phi({\mathcal K}_0) = {\mbox Res}_\phi [{\mathcal Z}] + dT_\phi,
$$
where ${\mbox Res}_\phi[{\mathcal Z}]$ is a `residue' form related to the 
zero set and $T_\phi$ is a transgression current.\index{transgression current}
This incorporates a generalized Poincar\'{e}-Lelong formula which
played a crucial role in the study of the quotient module in the rank one
case, see \cite{misra-DMV}.
\subsubsection{The Clebsch-Gorden formula}
Let $\mathcal M_1$ and $\mathcal M_2$ be
Hilbert spaces consisting of of holomorphic functions defined on $\Omega$ possessing reproducing kernels $K_1$ and $K_2,$ respectively. 
Assume that the natural action of $\mathbb C[\boldsymbol{z}]$ on
the Hilbert space $\mathcal M_1$ is continuous, that is, the map $(p, h) \to ph$ defines a
bounded operator on $\mathcal M$ for $p \in \mathbb C[\boldsymbol{z}].$ (We make no such assumption about
the Hilbert space $\mathcal M_2.$) Now, $\mathbb C[\boldsymbol{z}]$ acts naturally on the Hilbert space tensorproduct $\mathcal M_1\otimes \mathcal M_2$ via the map
$$(p, (h \otimes k)) \to p h \otimes k, p \in \mathbb C[\boldsymbol{z}],\,\, h\in\mathcal M_1,\,\, k \in \mathcal M_2.$$
\!\!\! The map $h \otimes k \to hk$ identifies the Hilbert space $\mathcal M_1\otimes \mathcal M_2$ as a reproducing
kernel Hilbert space of holomorphic functions on $\Omega\times \Omega.$ The module action
is then the point-wise multiplication $(p, hk) \to (ph)k$, where 
$$((ph)k)(\boldsymbol z_1, \boldsymbol z_2) = p(\boldsymbol z_1)h(\boldsymbol z_1)k(\boldsymbol z_2),\quad \boldsymbol z_1, \boldsymbol z_2 \in \Omega.$$

Let $\mathcal M$ be the Hilbert module $\mathcal M_1\otimes \mathcal M_2$
over $\mathbb C[\boldsymbol{z}]$. Let $\triangle \subseteq \Omega \times \Omega$ be the diagonal subset
$\{(z, z): z \in \Omega\}$ of $\Omega \times \Omega.$ Let $\mathcal S$ be the maximal submodule of functions in
$\mathcal M_1\otimes \mathcal M_2$ which vanish on $\triangle.$ Thus
{
$$0 \to \mathcal S \stackrel{X}{\to} \mathcal M_1\otimes \mathcal M_2
\stackrel{Y}{\to} \mathcal Q \to 0$$}
is a short exact sequence, where $\mathcal Q = (\mathcal M_1\otimes \mathcal M_2)/\mathcal S,$ $X$ is the inclusion map
and $Y$ is the natural quotient map.

As we have seen earlier in Section 1.4, the theorem of Aronszajn provides a complete description of the quotient module $\mathcal Q$ as the restrictions of functions in $\mathcal M$.
Now, let us investigate what happens if the submodule $\mathcal S$ is taken to be space  of functions vanishing to  order $2$ on $\triangle.$ Set  $\mathcal{H}_1$ and $\mathcal{H}_2$ to be the submodules defined by
$$\mathcal{H}_1=\{f\in(\mathcal{H}, K)\otimes (\mathcal{H}, K):f|_{\Delta}=0\}$$
and        
$$\mathcal{H}_2=\{f\in(\mathcal{H}, K)\otimes (\mathcal{H}, K):f|_{\Delta}=\partial_1f|_{\Delta}=\partial_2f|_{\Delta}=...=\partial_mf|_{\Delta}=0\}.$$
Let $\mathcal{H}_{11}=\mathcal{H}_2^\perp\ominus\mathcal{H}_1^\perp $
We have already described the quotient module $\mathcal H_{00}:=\mathcal M\ominus\mathcal H_1$, 
where $\mathcal M = (\mathcal H,K) \otimes (\mathcal H,K).$ This is the module $\mathcal M_{\rm res}.$ Set
$$\widetilde{K}(z,w)= \left(K^2(z,w)\partial_i\bar{\partial}_j 
\log K(z,w) \right )_{1\leq i,j\leq m}.$$
We claim that the function $\tilde{K}$ taking values in $\mathcal M_m(\mathbb{C})$ is a non-negative definite kernel.
To see this, set
$$\phi_i(w):= K_w\otimes\bar{\partial}_iK_w-\bar{\partial}_iK_w\otimes K_w, 
\quad 1\leq i\leq m,$$    
and note that each $\phi_i:\Omega \to \mathcal M$ is holomorphic. A simple calculation then shows that 
$$ \left( \! \left ( \langle\phi_j(w),\phi_i(z)\rangle_{\mathcal{M}} \right)\!\right ) = \widetilde{K}(z,w).$$
How to describe  the Hilbert space, or more importantly, the Hilbert module  $(\mathcal{H} ,\widetilde{K})$? May be, it is a quotient, or a sub-quotient module,  of the Hilbert module $(\mathcal H, K) \otimes (\mathcal H,K)?$ 

Let $\mathcal{H}_0$ be the subspace of $(\mathcal{H}, K)\otimes (\mathcal{H}, K)$ given by the smallest closed subspace containing the linear span of the vectors $\{\phi_i(w):w\in \Omega,1\leq i\leq m \}.$ 
From this definition,  it is not clear which functions belong to the subspace.  An  interesting limit computation given below shows that it coincides with $\mathcal H_{11}$. 

The point of what we have said so far is that we can explicitly describe the Hilbert modules $\mathcal{H}_2^\perp$and $\mathcal{H}_1^\perp $, up to an isomorphism of modules. Using the  
jet construction followed by the restriction map, one may also describe the direct sum 
$\mathcal{H}_2^\perp\oplus\mathcal{H}_1^\perp,$ again up to an isomorphism. 

But what is the module $\mathcal H_{11}$?  To answer this question (see \cite[Section 2.4.1]{misra-SG}), one must find the kernel function for $\mathcal H_{11}$. Setting $K_1$ to be the kernel function of the module $\mathcal H_1,$ indeed, we have 
 
$$
\lim_{{u\to z,}\\{v\to w}}\frac{K_1(z,u;v,w)}{(z-u)(\bar{w}-\bar{v})}
= \tfrac{1}{2} K(z,w)^2\partial\bar{\partial} \log K(z,w).
$$

This shows that $\mathcal H_{11}$ is isomorphic to $(\mathcal H, \tilde{K})$. The main challenge is to obtain an orthogonal decomposition of the Hilbert module $\mathcal M$ in the form of a composition series, namely, to complete the decomposition:  
$$\mathcal H \sim \mathcal H_{00} \oplus \mathcal H_{11} \oplus \mathcal H_{22} \oplus \cdots,$$
where $\mathcal H_{00}$ is the quotient module  $\mathcal M_{\rm res},$
as in the more familiar Clebsch-Gorden formula.\index{Clebsch-Gorden formula} 

\subsubsection{The sheaf model}For a Hilbert module $\mathcal M$ over a function algebra $\mathcal A(\Omega)$, not necessarily in the class $\mathrm{B}_1(\Omega)$, motivated by the correspondence of vector bundles with locally free
sheaf, we  construct a sheaf of modules $\mathcal S^{\mathcal M}(\Omega)$ over $\mathcal O(\Omega)$ corresponding to $\mathcal M$.\index{locally free sheaf}
We assume that $\mathcal M$ possesses all the properties for it to be in the class $\mathrm{B}_1(\Omega)$ except that the dimension of the joint kernel $\mathbb K(w)$ need not be constant.
We note that sheaf models have occured, as a very useful tool, in the study of analytic Hilbert modules (cf. \cite{misra-EP}).  Although, the model we describe  below is somewhat different.  

A Hilbert module $\mathcal M\subset \mathcal O(\Omega)$ is said to be in 
the class $\mathfrak{B}_1(\Omega)$ if  it possesses a reproducing kernel $K$ (we don't rule out the possibility: $K(w,w)=0$ for  $w$ in some closed subset $X$
of $\Omega$) and the dimension of $\mathcal M/\mathfrak m_w\mathcal M$ is finite for all $w\in \Omega.$

Most of the examples in $\mathfrak{B}_1(\Omega)$ arises in the form
of a submodule of some Hilbert module $\mathcal H (\subseteq \mathcal
O(\Omega))$ in the Cowen-Douglas class $\mathrm{B}_1(\Omega).$
Are there others?

Let $\mathcal S^{\mathcal M}(\Omega)$ be the subsheaf of the sheaf of holomorphic functions $\mathcal O(\Omega)$ whose stalk at $w \in \Omega$ is
$$
\big \{(f_1)_w \mathcal O_w + \cdots + (f_n)_w \mathcal O_w : f_1, \ldots , f_n \in \mathcal M \big \}, 
$$
or equivalently,
$$\mathcal S^{\mathcal M}(U) = \Big\{\sum_{i=1}^n \big ( {f_i}_{| U} \big) g_i: f_i \in \mathcal M, g_i \in \mathcal O(U) \Big \}$$
for $U$ open in $\Omega$.

\begin{proposition}
The sheaf ${\mathcal S}^{\mathcal M}(\Omega)$ is coherent.
\end{proposition}

\begin{proof}
The sheaf ${\mathcal S}^{\mathcal M}(\Omega)$ is generated by the set of functions 
$\{f: f\in \mathcal M\}$.   Let $\mathcal S_J^{\mathcal M}(\Omega)$ be the subsheaf generated by the set of functions 
$$J=\{f_1, \ldots , f_\ell\} \subseteq \mathcal M \subseteq \mathcal O(\Omega).$$ 
Thus $\mathcal S_J^{\mathcal M}(\Omega)$ is coherent.  An application of Noether's Lemma \cite{misra-GR} then guarantees that 
$$\mathcal S^{\mathcal M}(\Omega) = \cup_{J\, \mbox{\rm finite}} \mathcal S_J^{\mathcal M}(\Omega)$$ 
is coherent.
\end{proof}\index{coherent sheaf}

We note that the coherence of the sheaf implies, in particular,  that
the stalk $(\mathcal S^{\mathcal M})_w$ at $w\in \Omega$ is generated by a finite number of elements $g_1,  \ldots ,g_n$ from $\mathcal O(\Omega)$.

If $K$ is the reproducing kernel for $\mathcal M$ and $w_0\in \Omega$ is a fixed but arbitrary point, then for $w$ in a small neighborhood $U$ of $w_0$, we obtain the following decomposition theorem.
\begin{theorem}
Suppose $g^0_i,\,1\leq i \leq n,$ be a minimal set of generators for the stalk 
$(\mathcal S^{\mathcal M})_0:=(\mathcal S^{\mathcal M})_{w_0}$.
Then we have
$$
K(\cdot, w):=K_w = g^0_1(w)K^{(1)}_w + \cdots + g^0_n(w)K^{(n)}_w,  
$$
where $K^{(p)}:U \to \mathcal M$ defined by $w \mapsto K_w^{(p)}$,  $1 \leq p \leq n,$ is anti-holomorphic. Moreover, the elements $K^{(p)}_{w_0},\, 1\leq p \leq n$ are
linearly independent in $\mathcal M$, they are eigenvectors for the adjoint of the action of $\mathcal A(\Omega)$ on the Hilbert module $\mathcal M$ at $w_0$ and are uniquely determined by these generators.
\end{theorem}

We also point out that the Grammian 
$$G(w) = (\!(\langle {K^{(p)}_{w}}, {K^{(q)}_{w}}\rangle)\!)_{p,q=1}^n$$ 
is invertible in a small neighborhood of $w_0$ and is independent of the generators $g_1, \ldots, g_n$. Thus  \index{Grassmannian}
$$t: w \mapsto (K^{(1)}_{\bar{w}}, \ldots , K^{(n)}_{\bar{w}})$$ 
defines a holomorphic map into the Grassmannian $\mathrm G(\mathcal H,n)$ on the open set $U^*$.  The pull-back $E_0$ of the canonical bundle on $\mathrm G(\mathcal H,n)$ under this map then define a holomorphic Hermitian bundle on $U^*$.  Clearly, the decomposition of $K$ given in our Theorem is not canonical in anyway.  So, we can't expect the corresponding vector bundle $E_0$ to reflect the properties of the Hilbert module $\mathcal M$.  However, it is possible to obtain a canonical decomposition following the construction in \cite{misra-CS}. It then turns out that the equivalence class of the corresponding vector bundle $E_0$ obtained from this canonical decomposition is an invariant for the isomorphism class of the Hilbert module $\mathcal M$.  These invariants are by no means easy to compute.  At the end of this subsection, we indicate, how to construct invariants which are more easily computable. 
For now, the following Corollary to the decomposition theorem is immediate.

\begin{corollary}
The dimension of the joint kernel $\mathbb K(w)$ is greater or equal to the number of minimal generators of the stalk $(\mathcal S^{\mathcal M})_w$ at $w\in \Omega$.
\end{corollary}

Now is the appropriate time to raise a basic question. Let $\mathfrak m_w \subseteq \mathcal A(\Omega)$ be the maximal ideal of functions vanishing at $w$.  Since we have assumed $\mathfrak m_w \mathcal M$ is closed, it follows that the dimension of the joint kernel $\mathbb K(w)$ equals the dimension of the quotient module $\mathcal M/(\mathfrak m_w \mathcal M)$. However it is not clear if one may impose natural hypothesis on $\mathcal M$ to ensure
$$
\dim \mathcal M/(\mathfrak m_w \mathcal M) = \dim \mathbb K(w) 
= \dim (\mathcal S^{\mathcal M})_w/(\mathfrak m(\mathcal O_w) 
(\mathcal S^{\mathcal M})_w),
$$
where $\mathfrak m(\mathcal O_w)$ is the maximal ideal in $\mathcal O_w$, as well.

More generally, suppose $p_1, \ldots , p_n$ generate $\mathcal M$. Then $\dim \mathbb K(w) \leq n$ for all $w \in \Omega$.  If the common zero set $V$ of these is $\{0\}$ then $(p_1)_0, \ldots ,(p_n)_0$  need not be a minimal set of generators for $(\mathcal S^{\mathcal M})_0$. However, we show that they do if we assume $p_1, \ldots ,p_n$ are homogeneous of degree $k$, say. Further more, basis for  $\mathbb K(0)$ is the set of vectors:
$$\big \{p_1(\bar{\boldsymbol\partial})\}K(\cdot,w)_{|\,w=0}, \ldots ,  p_n(\bar{\boldsymbol\partial})\}K(\cdot,w)_{|\,w=0}\big \},$$
where $\bar{\boldsymbol\partial} =(\bar{\partial}_1, \ldots , \bar{\partial}_m)$.

Going back to the example of $H_0^2(\mathbb D^2)$, we see that it has
two generators, namely $z_1$ and $z_2$. Clearly, the joint kernel 
$$\mathbb K(w):= \ker D_{(M_1^*-\bar{w}_1, M_2^*-\bar{w}_2)}$$ 
at $w=(w_1,w_2)$ is spanned by
$$\{z_1 \otimes_{\mathcal A(\mathbb D^2)} 1_w, z_2 \otimes_{\mathcal A(\mathbb D^2)} 1_w \} = \{w_1 K_{H_0^2(\mathbb D^2)}(z,w),w_2 K_{H_0^2(\mathbb D^2)}(z,w) \}$$
which consists of two vectors that are linearly dependent except when $w=(0,0)$. We also easily verify that
$$
\big (\mathcal S^{H^2_0(\mathbb D^2)}\big )_w \cong \begin{cases} \mathcal O_w& w \not = (0,0)\\
\mathfrak m\,(\mathcal O_0) &w=(0,0)
                             .\end{cases}
$$
Since the reproducing kernel
\begin{eqnarray*}
K_{H_0^2(\mathbb D^2)}(z,w) &=& K_{H^2(\mathbb D^2)}(z,w) -1
= \frac{z_1\bar{w}_1 +z_2\bar{w}_2 - z_1 z_2\bar{w}_1 \bar{w}_2}{(1- z_1\bar{w}_1)(1-z_2\bar{w}_2)}, 
\end{eqnarray*}
we find there are several choices for $K^{(1)}_w$ and $K^{(2)}_w$, $w\in U$. However, all of these choices disappear if we set $\bar{w}_1 \theta_1 = \bar{w}_2$ for $w_1 \not = 0$, and take the limit:
$$
\lim_{(w_1,w_2) \to 0}  \frac{K_{H_0^2(\mathbb D^2)}(z,w)}{\bar{w}_1} =K^{(1)}_0(z) + \theta_1 K^{(2)}_0(z) = z_1 + \theta_1 z_2
$$
because $K^{(1)}_0$ and $K^{(2)}_0$ are uniquely deteremined by Theorem 1. Similarly, for $\bar{w}_2 \theta_2 = \bar{w}_1$ for $w_2 \not = 0$, we have$$
\lim_{(w_1,w_2) \to 0}  \frac{K_{H_0^2(\mathbb D^2)}(z,w)}{\bar{w}_2} =K^{(2)}_0(z) + \theta_2 K^{(1)}_0(z)= z_2 + \theta_2 z_1.
$$
Thus we have a Hermitian line bundle on the complex projective space $\mathbb P^1$ given by the frame $\theta_1 \mapsto z_1 + \theta_1 z_2$ and
$\theta_2 \mapsto z_2 + \theta_2 z_1$. The curvature of this line bundle  is then an invariant for the Hilbert module  $H_0^2(\mathbb D^2)$ as shown in \cite{misra-DMV}. This curvature is easily calculated and is given by the formula $\mathcal K(\theta) = (1+|\theta|^2)^{-2}$.  The decomposition theorem yields similar results in many other examples.

Let $\mathcal I$ be an ideal in the polynomial ring $\mathbb C[\mathbf{z}].$  The characteristic space of an ideal $\mathcal I$ in $\mathbb{C}[\mathbf{z}]$ 
at the point $w$ is the vector space
$$\mathbb V_w(\mathcal I):= \{q\in\mathbb{C}[\mathbf{z}] : q(D)p|_w=0,\, p\in\mathcal I\}.$$
The envolope $\mathcal I^e_w$ of the ideal $\mathcal I$ is 
$$\{p\in\mathbb{C}[\mathbf{z}] : q(D)p|_w=0,\,q\in\mathbb V_w(\mathcal I)\}.$$
If the zero set of the ideal $\mathcal I$ is $\{w\}$ then $\mathcal I^e_w = \mathbb V_w(\mathcal I).$

This describes an ideal by prescribing conditions on derivatives. We stretch this a little more. Let $\tilde{\mathbb V}_w(\mathcal I)$ be the auxiliary space 
$\mathbb V_w(\mathfrak m_w \mathcal I).$  We have 
$$\dim (\cap \mbox{\rm Ker} (M_j - w_j)^*) = \dim \tilde{\mathbb V}_w(\mathcal I) /\mathbb V_w(\mathcal I).$$
Actually, we have something much more substantial.
\begin{lemma}\label{nice}
Fix $w_0\in\Omega$ and polynomials $q_1,\ldots,q_t.$ Let $\mathcal I$ be a polynomial ideal and 
$K$ be the reproducing kernel corresponding the Hilbert module $[\mathcal I]$, which is assumed to be 
in $\mathfrak B_1(\Omega).$ Then the vectors
$$ q_1(\bar D)K(\cdot, w)|_{w=w_0},\ldots, q_t(\bar D)K(\cdot, w)|_{w=w_0} $$
form a basis of the joint kernel $\cap_{j=1}^m\ker (M_j - w_{0j})^*$
if and only if the classes $[q_1^*],\ldots,[q_t^*]$ form a basis of 
$\tilde{\mathbb V}_{w_0}(\mathcal I)/\mathbb V_{w_0}(\mathcal I).$
\end{lemma}
However, it is not clear if we can choose the polynomials $\{q_1, \ldots , q_t\}$ to be a generating set for the ideal $\mathcal I$. None the less, the following theorem produces a new set $\{q_1, \ldots , q_v\}$ of generators for $\mathcal I$, which is more or less \emph{``canonical''}.  Indeed, it is uniquely determined modulo linear transformations.
\begin{theorem}(\cite[Proposition 2.10]{misra-SBisM})\label{hi}
Let $\mathcal I\subset\mathbb{C}[\mathbf{z}]$ be  a homogeneous ideal and  
$\{p_1,\ldots,p_v\}$  be a minimal set of generators for $\mathcal I$ consisting of 
homogeneous polynomials. Let $K$ be the reproducing kernel corresponding to the Hilbert module 
$[\mathcal I]$, which is assumed to be in $\mathfrak B_1(\Omega).$ Then there exists a set of 
generators $q_1,...,q_v$ for the ideal $\mathcal I$ such that the set 
$$\{q_i(\bar D)K(\cdot,w)|_{w=0}:\,1\leq i \leq v\}$$
is a basis for the joint kernel $\displaystyle\cap_{j=1}^m \ker M_j^*.$
\end{theorem}

\subsection{The flag structure} \index{flag structure}
We have already discussed the class of operators $\mathcal FB_2(\Omega)$ which is contained in the Cowen-Douglas class $B_2(\Omega)$. A natural generalization is given below. The main point of introducing this smaller class is to show that a complete set of tractable unitary invariants
exist for this class. 

\begin{definition} We let $\mathcal{F}B_n(\Omega)$ be the set of all bounded linear operators $T$ defined on some complex separable Hilbert space $\mathcal H = \mathcal H_0 \oplus \cdots \oplus \mathcal H_{n-1},$ which are of the form
$$ T=\begin{pmatrix}
T_{0} & S_{0,1} & S_{0,2}&\cdots&S_{0,n-1}\\
0 &T_{1}&S_{1,2}&\cdots&S_{1,n-1} \\
\vdots&\ddots &\ddots&\ddots&\vdots\\
0&\cdots&0&T_{n-2}&S_{n-2,n-1}\\
0&\cdots&\cdots&0&T_{n-1}\\
\end{pmatrix},$$
where the operator $T_i:\mathcal H_i \to \mathcal H_i,$ defined on the complex separable Hilbert space $\mathcal H_i,$ $0\leq i \leq n-1,$ is assumed to be in $B_1(\Omega)$ and  $S_{i,i+1}:\mathcal H_{i+1} \to \mathcal H_i,$ is assumed to be a non-zero intertwining operator, namely, $T_iS_{i,i+1}=S_{i,i+1}T_{i+1},$  $0\leq i \leq n-2.$
\end{definition}
The set of operators in $\mathcal FB_n(\Omega)$, as before, also  belong to the Cowen-Douglas class $B_n(\Omega).$ To show this starting with the base case of $n=2,$ using induction, we don't even need  the intertwining condition. A rigidity theorem for this class of operators was proved in \cite[Theorem 3.5]{misra-JJKM}, which is reproduced below.

\begin{theorem}[Rigidity]\label{hdu}\index{rigidity}
Any two operators $T$ and $\tilde{T}$ in  $\mathcal{F}B_n(\Omega)$
are unitarily equivalent if and
only if there exists unitary operators $U_i$,  $0\leq i\leq n-1$,
such that $U_iT_i=\tilde{T}_iU_i$ and
$U_iS_{i,j}=\tilde{S}_{i,j}U_j,$  $i<j.$
\end{theorem}
From the rigidity, it is easy to extract unitary invariants for the operators in the class $\mathcal FB_n(\Omega)$, see \cite[Theorem 3.6]{misra-JJKM}. These invariants come from the intrinsic complex geometry inherent in the definition of the operators in the class $\mathcal FB_n(\Omega)$. Also, all the operators in this class are irreducible.  

\subsubsection{An application to module tensor products} 
\index{tensor product}\index{module tensor product}
In the early development of the quotient and sub-modules, it was expected that the localization using modules of rank $> 1$ might provide new insight. This was discussed in some detail in the paper \cite{misra-DMV} but the outcome was not conclusive.  The results from \cite[Section 4]{misra-JJKM} completes the study initiated in \cite{misra-DMV}.  Let $f$ be a polynomial in one variable. Set
$$\mathcal{J}_{\boldsymbol \mu}(f)(z)=\left(
                         \begin{array}{cccc}
                           \mu_{1,1}f(z) & 0 & \cdots & 0 \\
                           \mu_{2,1}\tfrac{\partial }{\partial z}f(z) & \mu_{2,2}f(z) & \cdots & 0 \\
                           \vdots & \vdots & \ddots & \vdots \\
                           \mu_{k,1}\tfrac{\partial^{k-1}}{\partial z^{k-1} }f(z) & \mu_{k-1,1}\tfrac{\partial^{k-2}}{\partial z^{k-2}}f (z)& \cdots &\mu_{k,k} f(z) \\
                         \end{array}
                       \right)
$$ 
where
$\boldsymbol{\mu} = \left( \mu_{i,j} \right)$ is a lower triangular matrix of complex numbers with $\mu_{i,i} =1,\,1\leq i \leq k.$
Lemma 4.1 of \cite{misra-JJKM} singles out those $\boldsymbol \mu$ which ensure $\mathcal{J}_{\boldsymbol \mu}(f g) = \mathcal{J}_{\boldsymbol \mu}(f) \mathcal{J}_{\boldsymbol \mu}(g)$. Fix one such $\boldsymbol \mu$. For $\mathbf x$ in $\mathbb{C}^k,$ and $f$ in the polynomial ring $\mathbb C[\boldsymbol z],$ define the module multiplication on $\mathbb C^k$  as follows: $$f\cdot \mathbf x=\mathcal J_{\boldsymbol \mu}(f)(w)\mathbf x.$$ The finite dimensional Hilbert space $\mathbb C^k$ equipped with this module multiplication is denoted by $\mathbb C^k_w[\boldsymbol \mu]$.

Let $\mathcal M$ be a Hilbert module possessing a reproducing kernel. We construct a module of $k$ - jets by setting
$$J\mathcal{M}=\Big \{\sum_{l=0}^{k-1}\tfrac{\partial^i}{\partial
z^i}{h}\otimes {\varepsilon}_{i+1}: h\in\mathcal{M}\Big \},$$
where $\varepsilon_{i+1}, \, 0 \leq i \leq k-1,$ are the standard basis vectors in $\mathbb C^k.$ Recall that there is a natural module action on $J\mathcal M,$ namely,
$$\Big (f, \sum_{l=0}^{k-1}\tfrac{\partial^i}{\partial z^i}{h} \Big )\mapsto
\mathcal{J}(f)\Big (\sum_{l=0}^{k-1}\tfrac{\partial^i}{\partial
z^i}{h}\otimes {\varepsilon}_{i+1}\Big ),\, f\in \mathbb C[z],\,
h\in\mathcal{M},$$
where
\begin{equation} \mathcal{J}(f)_{i,j} =  \begin{cases}  \binom{i-1}{j-1}
\partial^{i-j}f &\mbox{if} \,\, i\geq j,\\
0 & \mbox{otherwise}.
\end{cases}
\end{equation}
The module tensor product $J\mathcal{M}\otimes_{\mathbb C[z]}\mathbb{C}^k_w[\boldsymbol \mu]$ is easily identified with the quotient module $\mathcal{N}^{\bot},$ where $\mathcal N\subseteq \mathcal M$ is the sub-module spanned by the vectors
$$ \big\{ \sum_{l=1}^{k}(J_f\cdot{\bf
h}_l\otimes \varepsilon_l- {\bf h}_l\otimes
({\mathcal{J}_{\boldsymbol \mu}}(f))(w)\cdot\varepsilon_l): {\bf h}_l\in
J\mathcal{M},\varepsilon_l
\in \mathbb{C}^k, f\in \mathbb [z] \big\}.$$
\begin{theorem}
The Hilbert modules corresponding to the localizations $J\mathcal M\otimes_{\mathbb C[z]} \mathbb C^k_w[\boldsymbol \mu_i],$ $i=1,2,$ are in $\mathcal{F}B_k(\Omega)$ and they are isomorphic if and only if $\boldsymbol \mu_1 = \boldsymbol \mu_2.$
\end{theorem}

{\bf Acknowledgement.} 
The author thanks Soumitra Ghara, Dinesh Kumar Keshari,  Md. Ramiz Reza and Subrata Shyam Roy  for going through a preliminary draft of this article carefully and pointing out several typographical errors. The author gratefully acknowledges the financial support 
from the Department of Science and Technology in the form of the~J~C~Bose National 
Fellowship and from the University Grants Commission, Centre for Advanced Study.

\providecommand{\bysame}{\leavevmode\hbox to3em{\hrulefill}\thinspace}
\providecommand{\MR}{\relax\ifhmode\unskip\space\fi MR }
\providecommand{\MRhref}[2]{%
  \href{http://www.ams.org/mathscinet-getitem?mr=#1}{#2}
}
\providecommand{\href}[2]{#2}

\end{document}